\documentclass[12pt]{article}
\usepackage[english]{babel}
\usepackage[utf8]{inputenc}
\usepackage{amsthm,amsmath,amsfonts,amssymb}
\usepackage[normalem]{ulem}
\usepackage{nicefrac}
\usepackage{url}
\usepackage{mathtools}
\usepackage{color}
\usepackage{anysize}
\usepackage{geometry}
\usepackage{comment}
\marginsize{2cm}{2cm}{1,5cm}{1,5cm}

\numberwithin{equation}{section}
\theoremstyle{plain}

\newtheorem{theorem}{Theorem}[section]
\newtheorem{lemma}[theorem]{Lemma}
\newtheorem{proposition}[theorem]{Proposition}
\newtheorem{remark}[theorem]{Remark}

\newtheorem{definition}[theorem]{Definition}

\usepackage{amsmath}
\usepackage{graphicx}
\usepackage{booktabs}
\definecolor{myblue}{RGB}{20,60,160}
\definecolor{myred}{RGB}{200,16,40}
\usepackage[colorlinks = true,
            linkcolor = myblue,
            citecolor = myred 
            ]{hyperref}

\title{Limits of biconditioned Bienaym\'e--Galton--Watson trees}

\begin{document}
\author{Vanessa Dan\footnote{CMAP, École polytechnique, Institut Polytechnique de Paris, 91120 Palaiseau, France. \\ DMA, École Normale Supérieur, Université PSL, 75005 Paris, France. \\ vanessa.dan@polytechnique.edu.}}
\maketitle

\begin{abstract}
 We study the limiting behavior of a Bienaymé–Galton–Watson tree conditioned to have a large number of vertices and either a fixed number of leaves or a fixed number of internal nodes. The first biconditioning gives a universal result with respect to the offspring distribution. In contrast, the second case leads to a variety of limiting behaviors, ranging from condensation phenomena to more elongated tree structures, depending on the properties of the offspring distribution. To prove these results, we use tools from conditioned random walk theory and from analytic combinatorics. 
\end{abstract}
\tableofcontents
\section{Introduction}

In this article, we are interested in the asymptotic behavior of  Bienaymé–Galton–Watson trees (BGW trees) conditioned on having a large number of vertices with a fixed number of leaves or internal nodes. 

Scaling limits of BGW trees were initiated by Aldous \cite{A91,A92,A93} who proved that the rescaled critical BGW tree with finite-variance offspring distribution, conditioned on having a large total number of vertices, converges to a continuous random tree known as the Brownian Continuum Random Tree (CRT). The theoretical framework around this notion of convergence has since been well established (see the survey article \cite{LG05}). This convergence result has been extended to more general settings. In particular, Duquesne \cite{D03} (see also \cite{K13}) studied the case where the offspring distribution has infinite variance, and showed that the scaling limit of a critical Bienaymé–Galton–Watson tree whose offspring distribution lies in the domain of attraction of a stable law is a continuous random tree, known as the stable tree. It is worth noting that the case where the mean number of children is strictly greater than 1 is similar and leads to the same results.
However, if the mean is strictly less than 1, the geometry of a tree conditioned to be large can be very different and can lead to a condensation phenomenon \cite{K15}. Often motivated by applications to other random combinatorial structures, many other types of conditioning have also been considered, such as maximum height \cite{LG10}, maximum degree \cite{H17}, generation sizes \cite{ABD20}, or the total number of vertices with fixed degrees \cite{Kort13,R15,T20}.

While the asymptotic behavior of conditioned BGW trees is now well understood, many biconditioned cases (i.e. conditioning on both the total number of vertices and the number of leaves), remain largely unexplored. Nonetheless, several works have addressed specific biconditioned models, beginning with Labarbe and Marckert who studied the case of the uniform distribution, in which both the number of leaves $k_n$ and the number of internal nodes $n-k_n$ tend to infinity \cite{LM07}. They showed that, although the normalization factor must be adjusted depending on the number of imposed leaves, the scaling limit remains the same as in the single-conditioning case: the Brownian CRT. Then, Kargin considered BGW trees conditioned on both the total number of vertices $n$ and the number of leaves $k_n$, focusing on the regime where both $k_n$ and $n$ grow to infinity and $k_n=\alpha n+O(1)$, with $\alpha \in (0,1)$ \cite{Kargin23}. Under the assumption of exponential decay of the offspring distribution, he proved that the rescaled tree once again converges in distribution to the Brownian CRT. Finally, motivated by questions related to random maps, Kortchemski and Marzouk have recently obtained results \cite{KM23} on the scaling limits of the Łukasiewicz walk of such trees. The results in \cite{KM23} suggest that new scaling limits should emerge. 

In this article, we contribute to this line of research by studying the scaling limits of BGW trees under several new biconditioning regimes. We specifically focus on the case of biconditioning on the total number of vertices and either the number of leaves (Section \ref{sec3}) or the number of internal nodes (Sections \ref{sec5}$-$\ref{sec8}) is a fixed integer $k$ while the total number $n$ of vertices tends to infinity. 

\paragraph{Fixed number of leaves.} First, to ensure that the trees we consider are well-defined, we assume that the offspring distribution $\mu$ satisfies $\mu(1)>0$.
To study the biconditioned $\mu$-BGW tree,  we decompose the structure of a tree with $n$ vertices and $k$ leaves into two components: a reduced tree and a sequence of single-child ancestors. The reduced tree of $a$, denoted by $R(a)$, is obtained by removing all vertices in $a$ that have exactly one child. For each vertex $u$ in $R(a)$, we count the number of consecutive single-child vertices in $a$ that lie between $u$ and its parent in $R(a)$, or between $u$ and the root of $a$ if $u$ is the root of $R(a)$. We define the sequence of single-child ancestors of $a$, denoted by
$L(a)$, the sequence of these numbers when $R(a)$ is visited in lexicographic order. These two definitions are illustrated in Figure \ref{Fig1.2} below.

\begin{figure}[h!]
\centering
        \includegraphics[scale=1]{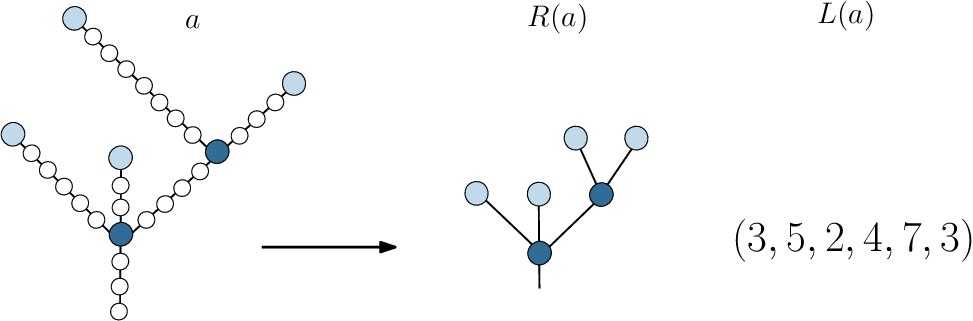}
        \caption{An example of a tree $a$ with its reduced tree $R(a)$ and its sequence of single-child ancestors $L(a)$.}
\label{Fig1.2}
\end{figure}
We will see in Section \ref{sec3} that $R(a)$ and $L(a)$ completely characterize a tree $a$ with $n$ vertices and $k$ leaves. Thus, studying the limit behavior of a biconditioned $\mu$-BGW tree is equivalent to studying the joint limit behavior of its reduced tree and its sequence of single-child ancestors.
\begin{theorem}\label{thcvg_feuilles-cas-binaire} Let $\mu$ be an offspring distribution with $\mu(1)>0$ and $\mu(2)>0$. Let $k\geq 1$ and $T_n^k$ a $\mu$-BGW tree conditioned to have $n$ vertices and $k$ leaves. Then, we have the following convergence in distribution:
\[\Bigl(R(T_n^k),\frac{L(T_n^k)}{n} \Bigl) \xrightarrow[n \rightarrow +\infty]{(d)} \bigl(R^k, \Delta \bigr)\]
where $R^k$ is a uniform random binary tree with $k$ leaves and $\Delta$ is a Dirichlet random variable with parameter $(1,\ldots,1)$ independent of $R^k$.
\end{theorem}
See Figure \ref{Fig0.1} for an illustration. 
Note that if $\mu(2)=0$, Theorem \ref{thcvg_feuilles-cas-binaire} cannot hold as this would imply that $\mathbb{P}(R^k_{n} = b) = 0$ for all binary trees $b$ with $k$ leaves . In that case, let $i_0 \geq 3$ be the smallest integer such that $\mu(i_0)>0$. It is worth noting that if $k-1$ is a multiple of $i_0-1$, we get that $R_{n}^k$ converges in distribution as $n$ tends to infinity to a uniform random $i_0$-ary tree with $k$ leaves. 
The main result of Section \ref{sec3}, Theorem \ref{R_nk_behavior}, extends Theorem \ref{thcvg_feuilles-cas-binaire} to any integer $k$.
\begin{figure}[h!]
\centering
        \includegraphics[scale=1]{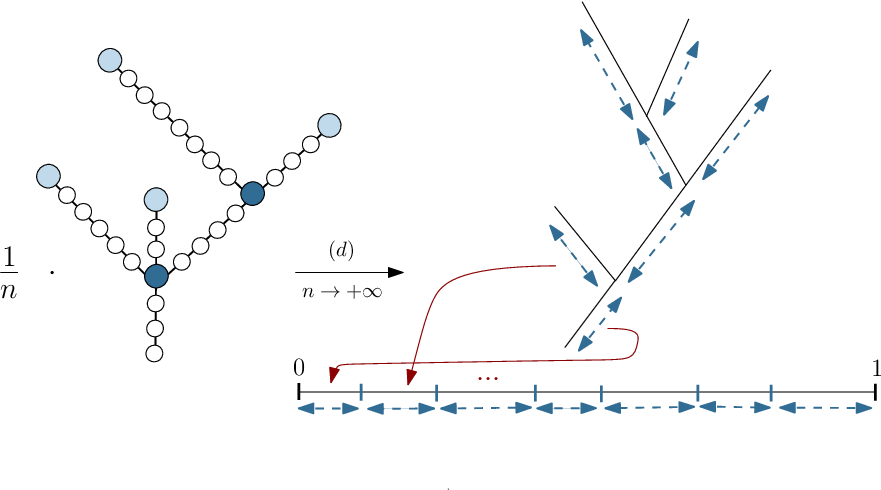}
        \caption{Illustration of the limit behavior of a $\mu$-BGW with $n$ nodes and $k$ leaves.}
\label{Fig0.1}
\end{figure}

\paragraph{Fixed number of internal nodes.} Unlike the first biconditioning, the limit here is not universal and depends sensibly on the properties of the offspring distribution $\mu$. As in the previous biconditioning, we begin by decomposing trees with $n$ vertices and $k$ internal nodes into two components: a reduced tree and a list of leaves. The reduced tree $R(a)$ of $a$ is the tree obtained by removing all leaves from $a$ and its sequence of leaves $L(a) \coloneqq(L_1(a),\ldots,L_{2k-1}(a))$  is the sequence of length $2k-1$ such that for every $i$ in $[\![1,2k-1]\!]$, the $i$-th element of the sequence equals the number of leaves of $a$ grafted in the $i$-th corner of $R(a)$ if the corner surrounds an internal node of $R(a)$ or that number minus one if the corner surrounds a leaf of $R(a)$. See Figure \ref{Fig0.2} for an illustration: Red vertices correspond to leaves in $R(a)$; in $L(a)$ they correspond to the number of leaves grafted on it in $a$ minus $1$.
\begin{figure}[h!]
\centering
        \includegraphics[scale=1]{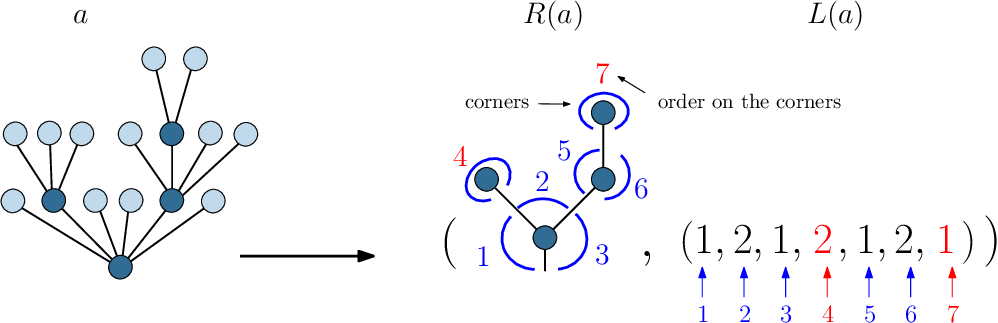}
        \caption{An example of a tree $a$ with its reduced tree $R(a)$ and its sequence of leaves $L(a)$.}
\label{Fig0.2}
\end{figure}

We denote by $T_{n,k}$ a $\mu$-BGW tree conditioned to have $n$ vertices and $k$ internal nodes.
We now describe the different cases in turn, according to the behavior of $\mu$.

The first case we consider is when the offspring distribution $\mu$ satisfies \eqref{hloc}: 
\begin{equation}\label{hloc_intro}
    ``\text{There exists } \ell \in \mathcal{R}_0 \text{ and } \beta > 1 \text{ such that for all } i \geq 0, \, \mu(i)=\ell(i)/i^{1+\beta}"
    \tag{$\mathcal{H}_{loc}$}
\end{equation} 
where $\mathcal{R}_0$ is the set of slowly varying functions (see Definition \ref{defslowlyvarfunc}). In this regime, a condensation phenomenon occurs around the root, meaning that most of the leaves are directly attached to it. We denote by $\ast_k$ the tree having $k-1$ leaves attached to a root.
\begin{theorem}\label{th1-intro}
     Under the assumption \eqref{hloc}, for every fixed $k\geq 1$ we have
\[d_{TV}\bigl(T_{n,k},D_{n,k} \bigr)\xrightarrow[n \rightarrow +\infty]{} 0,\]
where $d_{TV}$ denotes the total variation distance and $D_{n,k}$ is the random tree with $n$ vertices and $k$ internal nodes defined as follows: We first draw $k-1$ i.i.d. random variables $Z_1,\ldots,Z_{k-1}$ such that $\mathbb{P}(Z_1=j)=\mu(j)/(1-\mu(0))$ for all $j\geq 1$. We denote by $G_{n,k}$ the event `` $Z_1+\cdots+Z_{k-1} \leq n-k$''. On the event $G_{n,k}$, we let $D_{n,k}$ be the tree whose reduced tree is $\ast_k$ and whose sequence of leaves $L(D_{n,k})= (L_1(D_{n,k}),\ldots, L_{2k-1}(D_{n,k}))$ satisfies:
\begin{itemize}
         \item For $i \in [\![k+1,2k-1]\!]$, $L_{i}(D_{n,k})=Z_{i-k}$
         \item Conditionally given $(L_{k+1}(D_{n,k}),\ldots,L_{2k-1}(D_{n,k}))$, the vector $(L_1(D_{n,k}),\ldots, L_k(D_{n,k}))$ is a composition of $n-k-(\sum_{i=k+1}^{2k-1} L_i(D_{n,k}))$ into $k$ parts sampled uniformly at random.
\end{itemize}
Otherwise, $D_{n,k}$ is set to be an arbitrary fixed tree with $n$ vertices and $k$ leaves.
\end{theorem}
We give in Figure \ref{Fig0.3} a representation of $D_{n,k}$ on the event $G_{n,k}$.
\begin{figure}[h]
    \centering
    \includegraphics[scale = 1]{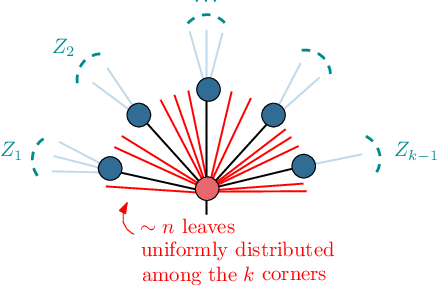}
    \caption{Illustration of $D_{n,k}$ on the event $G_{n,k}$.}
    \label{Fig0.3}
\end{figure}

We then shift our focus from a local setting to the following tail regime. Let $\mu$ be a probability distribution on $\mathbb{Z}_{\geq 0}$ with mean equal to $1+m$ and let $F_{\mu}$ be its generating function, defined by  \begin{equation} \label{defF}
 F_{\mu}(z) = z -mz +mz^2 + \ell\Bigl(\frac{1}{1-z}\Bigr)(1-z)^\alpha,   
\end{equation} where $m$ is a real number greater than $-1$, $\alpha$ a real number in $ (1,2)$ and $\ell :\mathbb{C}\rightarrow\mathbb{C}$ satisfies a technical assumption $\mathcal{H}$ (detailed in Section \ref{sec6}) and its restriction on $\mathbb{R}$ is slowly varying. We denote by $\mathcal{L}$ the set of such functions. As we will see in Section \ref{sec6}, these assumptions imply that $\mu([n,\infty[) \sim c_{\alpha} \ell(n)/n^{\alpha}$, where $c_{\alpha}$ a constant depending only on $\alpha$.
In this heavy-tail regime, a condensation phenomenon once again emerges around the root, illustrated in Figure \ref{Fig5}.
We denote by $c^{o}_i(a)$ is the number of children of the $i$-th internal vertex counted in lexicographic order in $a$ and $\mathbf{Dir}(1,\ldots,1)$ the Dirichlet distribution with parameter $(1,\ldots,1)$. 
\begin{theorem}\label{th3.2}
    Let $\mu$ be an offspring distribution with generating function satisfying \eqref{defF}, where $1< \alpha <2$, $m>-1$ and $\ell \in \mathcal{L}$. Then, we have the joint convergence:
    \begin{enumerate}
        \item $R(T_{n,k})$ converges in distribution to $\ast_k$
        \item $n^{-1} c^o_0(T_{n,k})\xrightarrow[n \rightarrow +\infty]{\mathbb{P}} 1$
        \item $n^{-1}(L_1(T_{n,k}),\ldots,L_k(T_{n,k})) \xrightarrow[n \rightarrow +\infty]{(d)} \mathbf{Dir}(1,\ldots,1) $
    \end{enumerate}
\begin{figure}[h!]
    \centering
    \includegraphics[scale = 1]{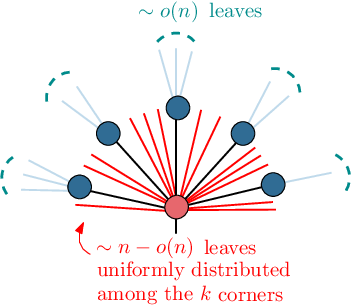}
    \caption{Illustration of the asymptotic shape of $T_{n,k}$.}
    \label{Fig5}
\end{figure}
\end{theorem}
The conclusions of Theorem \ref{th3.2} are weaker than that of Theorem \ref{th1-intro}, but the assumptions can be seen as more general (local assumption vs tail setting). We also observe that Theorem \ref{th3.2} addresses cases not covered by Theorem \ref{th1-intro}, and conversely. For instance, Theorem \ref{th3.2} allows $\mu$ to vanish an infinite number of times, while Theorem \ref{th1-intro} allows finite variance offspring distributions. Moreover, the two proofs rely on different techniques.

We then turn to a third regime, which generalizes the case of a geometric offspring distribution by imposing a transfer-type condition on the generating function of $\mu$. More precisely, we consider an offspring distribution $\mu$ whose generating function $F_{\mu}$ has a radius of convergence $\rho >1$, and satisfies the following assumption for some $\alpha > 0$:
\begin{equation}\label{halpha_intro}
``F_{\mu}(z)\text{ is } \Delta\text{-analytic } \text{and satisfies } F_{\mu}(z)\underset{z \to \rho}{\sim} \frac{c}{(1-z/\rho)^{\alpha}} \text{ for some }c \in \mathbb{R}".
    \tag{$\mathcal{H}_{\alpha}$}
\end{equation}
 In contrast to the previous cases, this regime leads to a fundamentally different limit, which we describe in terms of the reduced tree and the sequence of leaves. For any tree $a$, we denote by $c_u(a)$ the number of children of vertex $u$ in $a$.
\begin{theorem} \label{th_cvg_transfert_case-intro}
 Set $\alpha >0$. Let $\mu$ be an offspring distribution with generating function satisfying \eqref{halpha}. Then, we have:
\begin{enumerate}
    \item For every tree $a$ with $k$ vertices,
    \[\mathbb{P}(R(T_{n,k}) = a) \xrightarrow[n \rightarrow \infty]{}\frac{\prod_{u=1}^k w_u(a) }{\sum_{b \in \mathbb{T}_k} \prod_{u=1}^k w_u(b)  },\]
    where $w_u(b) \coloneqq  \Gamma(\alpha + c_u(b))/\Gamma(1+c_u(b))$ for $b \in \mathbb{T}_k$, $u \in [\![1,k]\!]$ and $\Gamma$ is Euler's gamma function.
    \item Conditionally given $ R_{n,k}=a$, we have the following convergence in distribution
    \[\frac{L(T_{n,k})}{n} \xrightarrow[n \rightarrow +\infty]{(d)} \mathbf{Dir}\biggl( \biggl(\frac{c_1(a)+\alpha}{c_1(a)+1}\biggr)_{1\leq i\leq c_1(a)+1},\ldots, \biggl(\frac{c_k(a)+\alpha}{c_k(a)+1}\biggr)_{1\leq i\leq c_k(a)+1}\biggr).\]
\end{enumerate}
\end{theorem}
In the particular case where $\alpha=1$, the reduced tree converges to a uniform tree with $k$ vertices, and the sequence of leaves, rescaled by $n$, converges in distribution to a Dirichlet distribution with parameters $(1,\ldots,1)$. A notable example of this setting is the geometric distribution on $\mathbb{Z}_{\geq0}$ with parameter $p \in (0,1)$. Its generating function is given by $\mu(k)=p(1-p)^k$ for $k\geq 0$ and  satisfies \eqref{halpha} for $\alpha = 1$, $c=p$, and $\rho = (1-p)^{-1}$. 

Finally, we consider a fourth and last regime, which can be seen as a generalization of the Poisson offspring distribution case, meaning that the generating function of $\mu$ satisfying the assumption \eqref{hp} defined by
\begin{equation}\label{hp-intro}
\text{``}F_{\mu}(z) = c\exp{(P(z))} \text{ where } c>0, P(z) \coloneqq \sum_{i=1}^pa_iz^i, a_i \geq 0 \text{ and } \mathrm{gcd}\{j:a_j\neq0\}=1. \text{''}
\tag{$\mathcal{H}_{P}$}
\end{equation}
\begin{theorem}\label{thPois-intro}
Let $\mu$ be an offspring distribution with generating function satisfying \eqref{hp}. 
\begin{enumerate}
    \item For every tree $a$ with $k$ vertices,
\[\mathbb{P}(R(T_{n,k}) = a) \xrightarrow[n \rightarrow \infty]{}\frac{\zeta_k(a)}{\sum_{b \in \mathbb{T}_k} \zeta_k(b)},\] where $\zeta_k(a) \coloneqq (\prod_{u=1}^k (c_u(a))!)^{-1} $. In other term, $R_{n,k}$ converges in distribution to a $\mathcal{P}(1)$-BGW tree conditioned to have $k$ vertices. 
\item Jointly, we have the following convergence in probability
    \[\frac{\bigl(X^n_1, \ldots, X^n_{k}\bigr)}{n} \xrightarrow[n \rightarrow +\infty]{(\mathbb{P})} \Bigl(\frac{1}{k},\ldots,\frac{1}{k}\Bigr),\]
    where $X^n_i$ is the number of children of the $i$-th internal node of $T_{n,k}$. 
    \item Conditionally given $ R(T_{n,k})=a$ and $(X^n_1, \ldots, X^n_{k})$, for every $i\in [\![1,k]\!] $, the vector \[\Bigl(L_{1+(\sum_{j=1}^{i-1}c_{j}(a)+1)}\bigl(T_{n,k}\bigr),\ldots, L_{c_i(a)+1+(\sum_{j=1}^{i-1}c_{j}(a)+1)}\bigl(T_{n,k}\bigr)\Bigr)\] is a composition of $n-k-(\sum_{j\neq i} X^n_j)$ into $c_i(a)+1$ parts sampled uniformly at random.
\end{enumerate}
\end{theorem}
In the particular case of a Poisson offspring distribution with parameter $\lambda >0$, given by $\mu(k)=e^{-\lambda}\lambda^k/k!$ for $k\geq 0$, the generating function satisfies \eqref{hp} for $ c = \exp{(-\lambda)} $ and $P(z)=\lambda z$, so does the generating function of $\sum_{i=1}^piX_i$ where $X_i \sim \mathcal{P}(a_i)$ are independent, $ c = \exp{(\sum_{i=1}^pa_i)} $ and $P(z)=\sum_{i=1}^pa_iz^i$.
\paragraph{Techniques.} We now comment on the main tools involved in the proof of each case. On the one hand, the main tool in the proof of Theorem \ref{thcvg_feuilles-cas-binaire}, its generalization Theorem \ref{R_nk_behavior} and Theorem \ref{th1-intro} is the encoding of conditioned BGW trees through their \L ukasiewicz  path, which, roughly speaking, correspond to a positive random walk conditioned to return to zero at a large time. 
In fact, the key idea in the proof of Theorem \ref{R_nk_behavior} is to show that the tree tends to maximize its number of edges under a given constraint. The \L ukasiewicz  path encoding makes this strategy explicit, enabling a direct application of classical results on conditioned random walks. In the proof of Theorem \ref{th1-intro}, we establish a condensation phenomenon by using a well-known result from the theory of random walks, often referred to as the `one big jump' theorem. This result essentially states that a random walk conditioned to reach a very high value in a short time typically does so by making one large jump and smaller ones. On the other hand, the key result that serves as a common starting point for the proofs of Theorems \ref{th3.2}, \ref{th_cvg_transfert_case-intro} and \ref{thPois-intro} is Lemma \ref{lemformuleRnk}, which expresses the distribution of the reduced tree in terms of the coefficients of the generating function and its derivatives. This leads us to study the asymptotic behavior of these coefficients, and thus this part relies on complex-analytic methods. 

\paragraph{Structure of the paper.} The paper is organized as follows. Section \ref{sec2} recalls basic definitions and properties of Bienaymé–Galton–Watson trees, as well as the coding function used to analyze their structure. Section \ref{sec3} introduces the biconditioning by the total number of vertices and leaves, and presents the main result along with its proof. Section \ref{sec4} presents the biconditioning by the total number of vertices and internal nodes and gives two different descriptions of a BGW tree, one through  the \L ukasiewicz path and the other via generating function. Finally, Sections  \ref{sec5}$-$\ref{sec8} examine different regimes depending on the behavior of the offspring distribution.

\paragraph{Acknowledgements.} I would like to sincerely thank my PhD advisors, Igor Kortchemski and Cyril Marzouk, for their invaluable advice, support, and feedback throughout this work.

\section{Background on rooted planar trees and BGW trees} \label{sec2}
In this section, we recall basic definitions and properties of trees and random trees that will be used throughout the paper. For further details and proofs of the stated results, the reader is referred to \cite{LG05, PC06}.
\subsection{Rooted planar trees and \L ukasiewicz paths}
Let $\mathcal{U} \coloneqq \bigcup_{n=0}^\infty (\mathbb{Z}_{>0})^n $ denote the set of words on the alphabet $\mathbb{Z}_{>0}$, where $\mathbb{Z}_{>0}^0 = \{\varnothing\} $. If $u,v \in \mathcal{U}, uv$ denotes the concatenation of $u$ and $v$. If $u,u',v \in \mathcal{U}$ and $u=vu'$, $v$ is said to be an ancestor of $u$. In particular if $u' \in \mathbb{Z}_{>0}$, $v$ is said to be the parent of $u$ and $u$ a child of $v$. 
\begin{definition}
    A rooted planar tree $\tau$ is a subset of $\mathcal{U}$ such that 
    \begin{enumerate}
        \item $\varnothing \in \tau$,
        \item if $u = vu' \in t$ then $v \in \tau$,
        \item for all $u \in \tau$, there exists $c_u(\tau) \in \mathbb{Z}_{\geq 0}$ such that $uj \in \tau$ if and only if $j \leq c_u(\tau)$; $c_u(\tau)$ is called the number of children of $u$ in $\tau$.
    \end{enumerate}
\end{definition}
See Figure \ref{Fig0} for an example.
\begin{figure}[h]
\centering
        \includegraphics[scale=1]{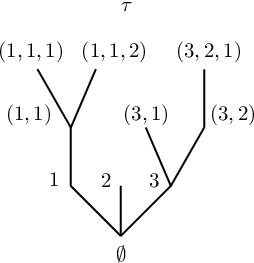}
        \caption{An example of a rooted planar tree.}
\label{Fig0}
\end{figure}

For every rooted planar tree $\tau$, $\varnothing$ is called the root of $\tau$, every $u \in \tau$ is a vertex (or a node) of $\tau$. If $u$ has no child it is a leaf of $\tau$ and it is an internal node otherwise. Finally, $|\tau|$ denotes the number of vertices in $\tau$. One can easily check that for every tree $\tau$,
$\sum_{u\in \tau}c_u(\tau)=|\tau|-1$. In the following, we will only consider rooted planar trees, so from now on we refer to them as 'trees' for simplicity. Let us denote by $\mathbb{T}$ the set of such trees and by $\mathbb{T}_n$ the set of trees having $n$ vertices.

At times we will explore a tree, meaning that we visit its vertices one by one. It is then necessary to choose the order in which the vertices are visited. A natural order on trees is given by the lexicographic order on the alphabet $\mathbb{Z}_{>0}\cup\{\varnothing\}$. 
Now, let us define a coding function of trees. 
\begin{definition} Let $\tau$ be a tree and $u_0,u_1,...,u_{|\tau|-1}$ be its vertices listed in the lexicographic order. The \L ukasiewicz path of $\tau$, denoted by $W(\tau)=(W_n(\tau), 0 \leq n \leq |\tau|)$ is defined as follows:
\[ W_n(\tau) = \left\{
    \begin{array}{ll}
        0 & \mbox{for } n=0 \\
        W_{n-1}(\tau) + c_{u_{n-1}} - 1 & \mbox{for all } 1\leq n \leq |\tau|.
    \end{array}
\right. \]
\end{definition}
See Figure \ref{Fig01} for an example.
\begin{figure}[h]
\centering
        \includegraphics[scale=1]{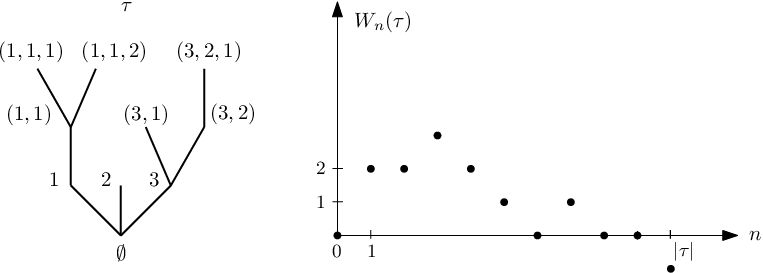}
        \caption{A tree $\tau$ and its \L ukasiewicz path.}
        \label{Fig01}
\end{figure}

The following well-known result establishes a bijection between finite trees and paths.

\begin{proposition} \label{bij_arbre_luka}
    The set $\mathbb{T}_n$ is in bijection with the set of paths $f:[\![0,n]\!] \rightarrow \mathbb{Z}$ such that \[ f(0)=0, \quad f(n)=-1, \quad f(i) \geq 0 \text{ and } f(i+1)-f(i) \in \mathbb{Z}_{\geq -1}  \text{ for } i \in [\![0,n-1]\!] .\]
\end{proposition}
\subsection{BGW trees}
From now on, we focus on a particular family of random trees,  Bienaymé–Galton–Watson trees. Roughly speaking, a  Bienaymé–Galton–Watson tree with offspring distribution $\mu$ is a random tree coding the genealogy of a population starting with one individual and where all individuals repoduce independently of each other according to the distribution $\mu$. Let us give a proper definition. 
\begin{definition}
    Let $\mu$ be an offspring distribution on $\mathbb{Z}_{\geq 0}$ and $(C_{u})_{u\in \mathcal{U}}$ be independent and indentically distributed random variables of law $\mu$. A random tree is a  Bienaymé–Galton–Watson tree with offspring distribution $\mu$ ($\mu$-BGW tree) if it has the same law as the random tree $T$ defined by the fact that for all $u \in T$, $c_u(T)=C_u$.
\end{definition}
In the sequel, we denote by $T$ a $\mu$-BGW tree and by $T_{n}$ such a tree conditioned on having $n$ vertices. One can note that the probability of $T$ being equal to a finite tree $\tau$ is explicit: 
\[\mathbb{P}(T=\tau)= \prod_{ u \in \tau}\mu(c_u(\tau)).\]

We define $W=(W_i, 0 \leq i \leq n)$ as the random walk started from $W_0=0$ with step-distribution $\nu$, where $\nu(i) =\mu(i+1) $ for all $ i \geq -1$. Let $H_{k}(W)\coloneqq\inf\{i\geq 0 : W_i=k\}$ denotes the first hitting time of $k \in \mathbb{Z}$ by the walk $W$.

The following proposition about the \L ukasiewicz path of a $\mu$-BGW tree directly follows from Proposition \ref{bij_arbre_luka}.
\begin{proposition}\label{propcodingluka}
  The \L ukasiewicz path $W(T_n)=(W_i(T_n), 0 \leq i \leq n)$ associated with $T_n$ has the same distribution as $W=(W_i, 0 \leq i \leq n)$ conditioned on $H_{-1}(W)=n$.
 \end{proposition}
  Since the conditioning on  $H_{-1}(W) = n$ can be delicate to work with, Proposition \ref{consequencecyclelemma} offers a practical alternative.

\begin{proposition}\label{consequencecyclelemma}
   Let $X_1,\ldots,X_n$ denotes the i.i.d. increments of $W=(W_i, 0 \leq i \leq n)$ and let $A \subset \mathbb{R}^n$ be stable by cyclic permutation, i.e. $\mathbf{x} \in A$ implies $\mathbf{x}^{(i)} \in A$ for all $0\leq i \leq n$. Then, we have the following equality:
   \[\mathbb{P}\bigl((X_1,...,X_n) \in A, W_n=-1, W_i \geq 0 \text{ for } i \in [\![1,n-1]\!]\ \bigr)= \frac{1}{n} \mathbb{P}\bigl((X_1,...,X_n) \in A, W_n=-1 \bigr).\]
\end{proposition}
This proposition is a direct consequence of a classical combinatorial result, known as the cycle lemma, recalled below.
\begin{lemma}[Cycle lemma]
    Set $\mathbf{x}=(x_1,\ldots,x_n)$ a sequence taking values in $\mathbb{Z}_{\geq -1}$ and such that $\sum_{i=1}^nx_i=-k$ for some $1\leq k\leq n$. For $1\leq i\leq n$, we define the $i^{th}$ cyclic shift of $\mathbf{x}$, denoted by $\mathbf{x}^{(i)}$, as follows:  $\mathbf{x}^{(i)} \coloneqq (x_{i+1},\ldots,x_n,x_1,\ldots,x_i)$. Note that $\mathbf{x}^{(n)}=\mathbf{x}$. Then, among the $n$ cyclic shifts, there are exactly $k$ of them that hit $-k$ for the first time at time $n$.
\end{lemma}

\section{Large BGW trees with a fixed number of leaves} \label{sec3}
\begin{table}[htbp]\caption{Table of the main notation and symbols introduced in Section \ref{sec3} and used later.}
\centering
\begin{tabular}{c c p{12cm} }
\toprule
$\mathbb{T}_{i}$ && the set of trees with $i$ vertices\\
$\mathbb{T}^{j}$ && the set of trees with $j$ leaves\\
$\mathbb{T}^j_{i}$ && the set of trees with $i$ vertices and $j$ leaves  \\
$\mathcal{L}^n_{m,p}$ && $ \{ (\ell_{1},\ldots,\ell_{m}) \in \mathbb{Z}_{\geq 0}^m : \forall i \in [\![1,m]\!] \text{ } \ell_{i} \geq 0 \text{ and } \sum_{i=1}^m \ell_{i} = n-p \}$ \\
\hline
$c_i(a)$ && the number of children of the $i$-th vertex, counted in lexicographic order, in the tree $a$ \\
$\phi_i(a)$ && $|\{u \in [\![1,k]\!] \text{ } | \text{ }c_u(a) = i \}|$ for a tree $a$ with $k$ vertices\\
$R(a)$ && the reduced tree of a tree $a$
\\
$L(a)$ && the sequence of single-child ancestors of a tree $a$ \\
\hline
$\nu(i)$  && $ \mu(i+1)$, $\forall i \geq -1 $ \\
$(X_i)_{i\geq0}$ && a sequence of i.i.d. random variables with distribution $\nu$ \\
$(Y_i)_{i\geq0}$ && a sequence of i.i.d. random variables such that $\mathbb{P}(Y_1=-1) = 0$ and $\forall j \geq 0,$ $\mathbb{P}(Y_1=j) = \nu(j)/(1-\mu(0))$ \\
$W_n$ && $ X_1 +\cdots+X_n$\\
$W'_n$ && $ Y_1 +\cdots+Y_n$\\
\hline
$T,R,L$ && a $\mu$-BGW tree, its reduced tree, its sequence of single-child ancestors  \\
$T^k_{n},R^k_{n}, L^k_{n}$ && a $\mu$-BGW tree conditioned to have $n$ vertices and $k$ internal nodes, $R(T^k_{n}), L(T^k_{n})$ \\
\bottomrule
\end{tabular}
\end{table}
The goal of this section is to study the asymptotic behavior of a $\mu$-BGW tree conditioned to have $n$ vertices and a fixed number of leaves $k \in \mathbb{Z}_{> 0}$, as $n \to \infty$, where the offspring distribution $\mu$ satisfies $\mu(1) > 0$.
We will prove a more general result than Theorem \ref{thcvg_feuilles-cas-binaire}. To do so, we begin by breaking down the structure of a tree with $n$ vertices and $k$ leaves into two components: a reduced tree and a sequence of single-child ancestors. We define these two concepts in the following subsection.

\subsection{Definitions}
 We denote by $T$ a $\mu$-BGW tree and by $T^k_{n}$ such a tree conditioned on having $n$ vertices and $k$ leaves (when this conditioning is non degenerate). In the sequel, $\mathbb{T}_{n}$ will be the set of trees having $n$ vertices, $\mathbb{T}^k$ the set of trees with $k$ leaves and $\mathbb{T}_{n}^k$ the set of trees having $n$ vertices and $k$ leaves. For $ a \in \mathbb{T}_{n}^k$, we denote by $c_u(a)$ the number of children of the $u$-th vertex counted in lexicographic order in $a$ and by $\phi_i(a)$ the size of the set $\{u \in [\![1,k]\!]: c_u(a) = i \}$.
 In the following, we assume that the root of a tree has a parent outgoing edge before its first child without a terminal vertex (such trees are sometimes called planted trees). This assumption makes the upcoming definitions more convenient. Finally, recall that we denote by $|a|$ the number of vertices in $a$ (which is also equal to the number of edges because $a$ is planted).

\begin{definition} The following definitions are illustrated in Figure \ref{Fig1.2}. Let $ a \in \mathbb{T}_{n}^k$.

The reduced tree of $a$, denoted by $R(a)$, is obtained by removing all vertices in $a$ that have exactly one child.

For each vertex $u$ in $R(a)$, we count the number of consecutive single-child vertices in $a$ that lie between $u$ and its parent in $R(a)$, or between $u$ and the root of $a$ if $u$ is the root of $R(a)$. We define the sequence of single-child ancestors of $a$, denoted by
$L(a)$, the sequence of these numbers when $R(a)$ is visited in lexicographic order.
Note that $L(a)$ is an element of $\mathcal{L}^n_{n-\phi_1(a),n-\phi_1(a)}$, where for $ m,p \geq 0, \mathcal{L}^n_{m,p} \coloneqq \bigl\{ (\ell_{1},\ldots,\ell_{m}) \in \mathbb{Z}_{\geq 0}^m : \forall i \in [\![1,m]\!] \text{ } \ell_{i} \geq 0 \text{ and } \sum_{i=1}^m \ell_{i} = n-p \bigr\}$.
\end{definition}

To simplify notation, for a $\mu$-BGW tree $T$ we will denote by $R$ its reduced tree and $L$ its sequence of single-child ancestors. Similarly, for a $\mu$-BGW conditioned to have $n$ vertices and $k$ leaves $T^k_{n}$ we will write $R^k_{n}\coloneqq R(T^k_{n})$ and $L^k_{n}\coloneqq L(T^k_{n})$.

\begin{remark}\label{rqbij1}
The function $a \mapsto (R(a),L(a))$ is a bijection between $\mathbb{T}^k_{n}$ and  $ \bigsqcup_{m=k+1}^{2k-1} \mathbb{T}^k_{m} \times \mathcal{L}^n_{m,m}$.
\end{remark}
One can observe that by Remark \ref{rqbij1}, $R(a)$ and $L(a)$ completely characterize a tree $a$ with $n$ vertices and $k$ leaves. Thus, studying the limit behavior of $T^k_{n}$ is equivalent to studying the joint limit behavior of  $ R^k_{n} $ and $ L^k_{n}$. 

\subsection{Main result}
We denote by $\mathbf{Dir}(1,\ldots,1)$ the Dirichlet distribution with parameter $(1,\ldots,1)$. For $m\in \mathbb{Z}_{\geq2}$, let $ (V_{1}, \ldots, V_{m-1}) $ be the increasing rearrangement of $m-1$ i.i.d. uniform random variables on $[0,1]$. It is well known that  $(V_{1}, V_{2}-V_{1}, \ldots , V_{m-1}-V_{m-2}, 1-V_{m-1}) $ follows the Dirichlet distribution with parameter $(1, \ldots , 1)$. We also recall a classic result which will be used later.

\begin{lemma}\label{thcvgdir}
     Let $m\in \mathbb{Z}_{\geq2}$ and $p\in \mathbb{Z}_{\geq0}$. If $(u_{n,1},\ldots, u_{n,m})$ is uniformly distributed on $\mathcal{L}^n_{m,p}$, then we have the following convergence in distribution \[\frac{\bigl(u_{n,1}, \ldots , u_{n,m}\bigr)}{n} \xrightarrow[n \rightarrow \infty]{(d)} \mathbf{Dir}(1,\ldots,1).\]
\end{lemma}

Before stating the theorem on the asymptotic behavior of $ R^k_{n} $ and $ L^k_{n}$, let us introduce some notation. We denote by $\mathrm{Supp}(\mu) \coloneqq \{ i \geq 0 : \mu(i) >0 \}$ and by $\mathcal{K}(\mu)$ the set of integers defined as follows \[\mathcal{K}(\mu) \coloneqq \left\{ 1+ \sum_{\substack{i \in \mathrm{Supp}(\mu)\\
              i>0}} b_i(i-1) : b_i \in \mathbb{Z}_{\geq 0} \right\}.\]
Note that a $\mu$-BGW tree conditioned to have $k$ leaves is well-defined if and only if $k \in \mathcal{K}(\mu)$. Among these trees, we will consider those that are maximal in the following sense: $a \in \mathbb{T}^k$ is said to be maximal if $a$ has $b_j$ vertices with $j$ children for $j \in \mathrm{Supp}(\mu)\cap \mathbb{Z}_{\geq 1}$, where $(b_j)_{j \in \mathrm{Supp}(\mu)}$ maximizes the sum $\sum_{j \in \mathrm{Supp}(\mu)} b_j$ under the constraint $(\mathcal{C})$ defined by
\begin{equation*}\label{constraint}
    (\mathcal{C}) \coloneqq``b_1=0 \quad \text{and} \sum_{\substack{i \in \mathrm{Supp}(\mu)\\
              i>0 }}  b_i(i-1) = k-1".
\end{equation*}
Observe that every tree $a \in \mathbb{T}^k$ has $b_0=k$ vertices with no children. We denote by $\mathcal{B}^{\mathrm{max}}$ the set of sequences $\mathbf{b}=(b_j)_{j \in \mathrm{Supp}(\mu)}$ that maximizes $\sum_{j \in \mathrm{Supp}(\mu)} b_j$ under the constraint $(\mathcal{C})$. Such maximizing sequences exists, since only a finite number of sequences satisfy $(\mathcal{C})$. Note that for every $\mathbf{b} \in \mathcal{B}^{\mathrm{max}}$, $b_j =0$ for $j \geq k$. 
For all  $\mathbf{b} \in \mathcal{B}^{\mathrm{max}}$, we denote by $\mathbb{T}^k(\mu, \mathbf{b})$ the set of trees having $k$ leaves and $b_j$ vertices with $j$ children for every $j \in \mathrm{Supp}(\mu)\cap \mathbb{Z}_{\geq 1}$. We also define $\mathbb{T}^k(\mu)$ as the set of all maximal trees \[\mathbb{T}^k(\mu) \coloneqq \bigsqcup_{\mathbf{b} \in \mathcal{B}^{\mathrm{max}}}  \mathbb{T}^k(\mu, \mathbf{b}).\]

For example, if $\mu(2)>0$ then $\mathcal{K}(\mu)=\mathbb{Z}_{\geq 0}$ and $\mathbb{T}^k(\mu) = \{ \text{binary trees with } k \text{ leaves}\}$. If $\mu(2)=0$, let $i_0 \geq 3$ be the smallest integer such that $\mu(i_0)>0$, hence  as $ \mathbb{Z}_{\geq 0}(i_0-1)+1 \subset \mathcal{K}(\mu)$, we deduce that if $k \in  \mathbb{Z}_{\geq 0}(i_0-1)+1 $ then $\mathbb{T}^k(\mu) = \{ i_0\text{-ary trees with } k \text{ leaves}\}$. 

We now state the theorem describing the limiting behavior of  $ R^k_{n} $ and $ L^k_{n}$.

\begin{theorem}\label{R_nk_behavior} Let $\mu$ be an offspring distribution with $\mu(1)>0$ and fix $k \geq2$ with $k \in \mathcal{K}(\mu)$. Then, we have the following convergence in distribution:
\[\Bigl(R_n^k,\frac{L_n^k}{n} \Bigl) \xrightarrow[n \rightarrow +\infty]{(d)} \bigl(R^k, \Delta \bigr)\]
where $R^k$ is a $\mu$-BGW tree conditioned to lie in  $\mathbb{T}^k(\mu)$ and $\Delta$ is a Dirichlet random variable with parameter $(1,\ldots,1)$ independent of $R^k$.
\end{theorem}

As a result of the previous discussion, we can mention two interesting particular cases of the first part of Theorem \ref{R_nk_behavior}: If $\mu$ is an offspring distribution with $\mu(2)>0$, then $R_{n}^k$ converges in distribution as $n \rightarrow \infty$ to a uniform binary tree with $k$ leaves (this is Theorem \ref{thcvg_feuilles-cas-binaire}).
Moreover, it is clear that if $\mu(2)=0$ this cannot hold, as this would imply that $\mathbb{P}(R^k_{n} = b) = 0$ for all binary trees $b$. If $\mu(2)=0$, let $i_0 \geq 3$ be the smallest integer such that $\mu(i_0)>0$. If $k \in \mathbb{Z}_{\geq 0}(i_0-1)+1$, then $R_{n}^k$ converges in distribution as $n$ tends to infinity to a uniform random $i_0$-ary tree with $k$ leaves.

\begin{remark}
    One possible approach to prove the convergence of $R_n^k$ in Theorem \ref{R_nk_behavior} would be to show that the ratio $\mathbb{P}(R_n^k=\tau_1)/\mathbb{P}(R_n^k=\tau_2)$ tends to $0$ when $\tau_2 \in \mathbb{T}^k(\mu)$ but $\tau_1 \notin \mathbb{T}^k(\mu)$ and converges to $\prod_{j = 2}^k\mu(j)^{\phi_j(\tau_1)}/\prod_{j = 2}^k\mu(j)^{\phi_j(\tau_2)}$ when $\tau_1,\tau_2 \in \mathbb{T}^k(\mu)$. We nevertheless choose a different approach for two main reasons. First, the proof presented below relies on the coding by the \L ukasiewicz path, a technique that will also be used later in the proof of Theorem \ref{th1-intro}.
Second, this method yields an asymptotic equivalent for $\mathbb{P}(T \in \mathbb{T}_n^k)$, see Equation \eqref{eqT}, which is also interesting on its own.
\end{remark}
\begin{proof} We begin by proving that $R_n^k$ converges in distribution as $n$ tends to infinity to a uniform random tree in $\mathbb{T}^k(\mu)$.
Fix $a \in \mathbb{T}^k(\mu)$. Then, $\phi_1(a) = 0$ and we have 
\begin{equation}\label{formuleR_nk}
\mathbb{P}(R^k_{n} = a)=\frac{\mathbb{P}(R = a , T \in \mathbb{T}^k_{n})}{\mathbb{P}(T \in \mathbb{T}^k_{n})} = \binom{n-1}{|a|-1} \frac{\mu(0)^{k}\mu(1)^{n-|a|}}{\mathbb{P}(T \in \mathbb{T}^k_{n})}  \prod_{j = 2}^k\mu(j)^{\phi_j(a)}.
\end{equation}
Indeed, the number of trees with $n$ vertices and $k$ leaves that have $a$ as their reduced tree is equal to the number of ways to fill the $|a|$ edges of $a$ with $n-|a|$ vertices with a single child, which is equal to $\binom{n-1}{|a|-1}$.
By using the coding of $T$ by its \L ukasiewicz path (Proposition \ref{propcodingluka}), if $W$
denotes the random walk starded from $W_0=0$ and with i.i.d. increments $(X_{i})_{i \in \mathbb{Z}_{\geq 0}}$ distributed according to $\nu$ defined by $ \nu(i) =\mu(i+1) $ for all $ i \geq -1$, then $\mathbb{P}(T \in \mathbb{T}^k_{n})$ is equal to \[\mathbb{P}\Bigl(W_{n}=-1,  W_{i} \geq 0 \,\forall i \in [\![1,n-1]\!], \bigl|\{i \in [\![0,n-1]\!] \text{: }  X_{i}=-1\}\bigr|=k\Bigr).\]
Then, applying Proposition \ref{consequencecyclelemma}, we deduce that this probability is equal to $ \mathbb{P}(W_{n}=-1, |\{i \text{ : } X_i =-1\}|= k) / n$, which is equal to
\begin{align*}
&\frac{1}{n}\sum_{\substack{ I \subset [\![1,n]\!]\\
              |I|=k}}  \mathbb{P}\Biggl( X_i = -1 \text{ } \forall i \in I , \sum_{j \notin I} X_j = k-1, X_i \neq -1 \text{ } \forall i \notin I \Biggr) \\
    & = \frac{1}{n}\sum_{\substack{ I \subset [\![1,n]\!]\\
              |I|=k}}  \mathbb{P}\Bigl( X_i = -1 \text{ } \forall i \in I \Bigr) \mathbb{P}\Biggl(\sum_{j \notin I} X_j = k-1, X_i \neq  -1 \text{ } \forall i \notin I \Biggr) \\
    & = \frac{1}{n}\binom{n}{k} \mu(0)^k (1-\mu(0))^{n-k} \mathbb{P}\Biggl(\sum_{j =1}^{n-k} X_j = k-1 \,\Big|\, X_i \neq  -1 \text{ } \forall 1\leq i \leq n-k \Biggr),
\end{align*}
which is equal to $\mathbb{P}(B_{n}=k)\mathbb{P}(W'_{n-k} = k-1)/n$, where $B_{n}$ is the sum of $n$ independent Bernoulli random variables with parameter $\mu(0)$ and $W'$ is a random walk starting at $0$ with i.i.d. increments $(Y_{i})_{i \in \mathbb{Z}_{\geq 0}}$ such that $\mathbb{P}(Y_1 = i) =\nu(i)/(1-\mu(0))$ for all $ i \geq 0$. In other words, $W'$ is equal in distribution to the random walk $W$ conditioned to have only non negative steps. Thus, we have \[\mathbb{P}\bigl(T \in \mathbb{T}^k_{n}\bigr) = \frac{1}{n}\mathbb{P}\bigl(B_{n}=k\bigr)\mathbb{P}\bigl(W'_{n-k} = k-1\bigr).\]
By summing over all possible increments, we get that $\mathbb{P}(W'_{n-k} = k-1)$ is equal to
\[\sum_{\substack{0\leq y_{1},\ldots,y_{n-k} \leq k-1\\ y_{1}+\cdots+y_{n-k}=k-1}} \mathbb{P}(Y_{1} = y_{1},\ldots,Y_{n-k} = y_{n-k}) = \frac{1}{(1-\mu(0))^{n-k}} \sum_{\substack{0\leq y_{1},\ldots,y_{n-k} \leq k-1\\ y_{1}+\cdots+y_{n-k}=k-1}} \mu(y_{1}+1)\cdots\mu(y_{n-k}+1).\]
Therefore, we have
\begin{equation}\label{eqRfeuilles}
     \mathbb{P}(T \in \mathbb{T}^k_{n})= \frac{\mu(0)^k}{n} \binom{n}{k} \sum_{\substack{0\leq y_{1},\ldots,y_{n-k} \leq k-1\\ y_{1}+\cdots+y_{n-k}=k-1}} \mu(y_{1}+1)\cdots\mu(y_{n-k}+1).
\end{equation}
 The sum appearing in Equation \eqref{eqRfeuilles} is equal to
\begin{align*}
& \sum_{\substack{b_{1}+\cdots+b_{k}=n-k\\
               2b_{2}+3b_{3}+\cdots+kb_{k}=n-1}} \binom{n-k}{b_{1},\ldots,b_{k}} \mu(1)^{b_{1}}\cdots\mu(k)^{b_{k}} \\ 
               &=  \sum_{p=1}^{k-1}\sum_{\substack{b_2+\cdots+b_{k}=p\\
               b_{2}+2b_{3}+\cdots+(k-1)b_{k}=k-1}} \binom{n-k}{n-k-p, b_2,\ldots,b_{k}} \mu(1)^{n-k-p} \mu(2)^{b_2}\cdots\mu(k)^{b_{k}} \\
               &= \sum_{p=1}^{k-1}\sum_{\substack{(b_j)_{j\in \mathrm{Supp}(\mu)\cap [\![2,k]\!]} \\\sum b_j=p\\
               \sum (j-1)b_j=k-1}} \binom{n-k}{n-k-p, b_2,\ldots,b_{k}} \mu(1)^{n-k-p} \mu(2)^{b_2}\cdots\mu(k)^{b_{k}}
\end{align*}
where $\binom{n-k}{n-k-\sum b_j, b_2,\ldots,b_{k}}$ is asymptotically equivalent to $(b_2!\cdots b_{k}!)^{-1}n^{b_2+\cdots+b_k}$ and all the sums have a finite number of terms. Let $\mathbf{b}\coloneqq(b_j)_{j \in \mathrm{Supp}(\mu)} \in \mathcal{B}^{\mathrm{max}}$ and $p_{\mathrm{max}} \coloneqq \sum_{j\geq 2} b_j$. Then, we get the following equivalent:
\begin{equation*}
\sum_{\substack{0\leq y_{1},\ldots,y_{n-k} \leq k-1\\ y_{1}+\cdots+y_{n-k}=k-1}} \mu(y_{1}+1)\cdots\mu(y_{n-k}+1)\underset{n \to +\infty}{\sim} \sum_{\mathbf{b} \in \mathcal{B}^{max}}\frac{n^{p_{\mathrm{max}}}}{b_2! \cdots b_{k}!}\mu(1)^{n-k-p_{\mathrm{max}}} \prod_{j=2}^k \mu(j)^{b_j}.
\end{equation*}
 Recall that the number of vertices in $a$ is equal to $|a| = k+p_{\mathrm{max}}$ for $a \in \mathbb{T}^k(\mu)$. Then, using Equation \eqref{eqRfeuilles}, we get that
\begin{equation} \label{eqT}
    \mathbb{P}(T \in \mathbb{T}^k_{n}) \underset{n \to +\infty}{\sim} \frac{\mu(0)^k}{n}\binom{n}{k} n^{p_{\mathrm{max}}} \mu(1)^{n-k-p_{\mathrm{max}}} \sum_{\mathbf{b} \in \mathcal{B}^{\mathrm{max}}} \frac{1}{b_2! \cdots b_{k}!}\prod_{j=2}^k \mu(j)^{b_j}.
\end{equation}
Now, combining Equations \eqref{formuleR_nk} and \eqref{eqT}, we obtain the following equivalent

\begin{equation} \label{eqR}
    \mathbb{P}(R^k_n =a) \underset{n \to +\infty}{\sim} \frac{k!}{(k+p_{\mathrm{max}}-1)!} \frac{\prod_{j=2}^k \mu(j)^{\phi_j(a)}}{\sum_{\mathbf{b} \in \mathcal{B}^{\mathrm{max}}} \frac{1}{b_2! \cdots b_{k}!}\prod_{j=2}^k \mu(j)^{b_j}}.
\end{equation}
Moreover by Equation (6.19) of \cite{PC06} giving the number of plane trees with prescribed degrees $\mathbf{b}$ and $k$ leaves, we have that \[\bigl| \mathbb{T}^k(\mu,\mathbf{b})\bigr| = \frac{(k+p_{{\mathrm{max}}}-1)!}{k!b_2! \cdots b_{k}!}.\] 
Consequently, we can rewrite the equivalent in Equation \eqref{eqR} as follows:
\begin{equation} \label{eqR}
    \mathbb{P}(R^k_n =a) \underset{n \to +\infty}{\sim}  \frac{\prod_{j=2}^k \mu(j)^{\phi_j(a)}}{\sum_{\mathbf{b} \in \mathcal{B}^{\mathrm{max}}} \bigl| \mathbb{T}(\mu,\mathbf{b})\bigr|\prod_{j=2}^k  \mu(j)^{b_j}},
\end{equation}
and conclude that $R^k_n$ converges in distribution as $n$ tends to infinity to a $\mu$-BGW tree conditioned to lie in $\mathbb{T}^k(\mu)$.

Now, let us deduce Theorem \ref{R_nk_behavior}. Set $a \in \mathbb{T}^k(\mu)$. From the preceding, it is enough to prove that conditionally on the event $\{R^k_{n}=a\}$, $L^k_{n}$ has the uniform distribution on $\mathcal{L}^n_{|a|,|a|}$ and then apply Lemma \ref{thcvgdir} to obtain the desired result. Fix $\ell = (\ell_{1},\ldots,\ell_{|a|}) \in \mathcal{L}^n_{|a|,|a|}$, and let us prove that \[ \mathbb{P}\bigl(L^k_{n}= \ell \,\big| \, R^k_{n}=a \bigr) = \frac{1}{|\mathcal{L}^n_{|a|,|a|}|} = \frac{1}{\binom{n-1}{|a|-1}}.\]
On the one hand, by Equation \eqref{formuleR_nk}, we have that $\mathbb{P}(R^k_{n}= a)$ is equal to \[ \binom{n-1}{|a|-1} \frac{\mu(0)^{k}\mu(1)^{n-|a|}}{\mathbb{P}(T \in \mathbb{T}^k_{n})}  \prod_{i = 2}^k\mu(i)^{\phi_i(a)},\] on the other hand we have that $\mathbb{P}(L^k_n= \ell , R^k_{n} = a )$ is equal to \[ \frac{\mathbb{P}(L=\ell,R=a,T \in \mathbb{T}^k_{n})}{\mathbb{P}(T \in \mathbb{T}^k_{n})} = \frac{\mu(0)^{k}\mu(1)^{n-|a|}}{\mathbb{P}(T \in \mathbb{T}^k_{n})}  \prod_{i = 2}^k\mu(i)^{\phi_i(a)},\]
so the proof is complete using $\mathbb{P}(L^k_{n}= \ell \mid R^k_{n}=a) = \mathbb{P}(L^k_{n}= \ell, R^k_{n}=a)/\mathbb{P}(R^k_{n}=a)$.
\end{proof}

\section{BGW trees with a fixed number of internal nodes}\label{sec4}
In this section, we focus on BGW trees with $n$ vertices and $k$ internal nodes, where both $n$ and $k$ are fixed. In Sections \ref{sec5}$-$\ref{sec8}, we will then use these results to study the limiting behavior of BGW trees as $n$ tends to infinity. 
\subsection{Definitions}
\begin{table}[htbp]\caption{Table of the main notation and symbols introduced in Subsection $4.1$ and used later.}
\centering
\begin{tabular}{c c p{12cm} }
\toprule
$\mathbb{T}_{i}$ && the set of trees with $i$ vertices\\
$\mathbb{T}_{i,j}$ && the set of trees with $i$ vertices and $j$ internal nodes  \\
\hline
$R(a), L(a)$ && the reduced tree of a tree $a$, the sequence of leaves of $a$ \\
$T,R,L$ && a $\mu$-BGW tree, its reduced tree, its sequence of leaves \\
$T_{n,k}, R_{n,k}, L_{n,k} $ && the $\mu$-BGW tree conditioned to have $n$ vertices and $k$ internal nodes, $ R(T_{n,k})$, $ L(T_{n,k})$  \\
\bottomrule
\end{tabular}
\end{table}
 We denote by $T$ a $\mu$-BGW tree and by $T_{n,k}$ such a tree conditioned on having $n$ vertices and $k$ internal vertices (when this conditioning is non degenerate). From now on, $\mathbb{T}_{n,k}$ will be the set of trees having $n$ vertices and $k$ internal nodes and we recall that $\mathbb{T}_{k}$ is the set of trees with $k$ vertices. Recall that for $ a \in \mathbb{T}_{k}$, we denote by $c_u(a)$ the number of children of the $u$-th vertex counted in lexicographic order in $a$ and by $\phi_i(a)$ the size of the set $\{u \in [\![1,k]\!]: c_u(a) = i \}$. As in the previous section, we assume that the root of a tree has a parent edge before its first child.
 
 Following the same approach as in Section \ref{sec3}, we begin by decomposing the structure of a tree in $\mathbb{T}_{n,k}$ into two components: a reduced tree and a sequence of leaves, as illustrated in Figure \ref{Fig0.2}.
 
 \begin{definition} Let $ a \in \mathbb{T}_{n,k}$. The reduced tree of $a$, denoted by $R(a)$, is the tree obtained by removing all leaves from $a$. Note that $R(a) \in \mathbb{T}_{k}$.
\end{definition}

To simplify notation, for a $\mu$-BGW tree $T$ we will denote by $R$ its reduced tree and for a $\mu$-BGW conditioned to have $n$ vertices and $k$ internal nodes $T_{n,k}$ we will write $R_{n,k}$ for its reduced tree.

\begin{definition}
    Let $ a \in \mathbb{T}_{n,k}$. The internal corners of $a$ are the angular sectors defined by the parent edges of the internal nodes. Note that the number of internal corners in $a$ is equal to $\sum_{u \in R(a)}(c_u(a)+1) = 2k-1$.
\end{definition}

We have to choose an order on the internal corners of a tree to be able to list them. The chosen order is as follows: We visit the internal nodes of the tree in lexicographical order, starting with the first internal node: the root. We visit the internal corners of the root in the clockwise direction, starting from the parent edge. Once finished, we move on to the second internal node. We repeat this process until all internal nodes have been seen (Figure \ref{Fig0.2} gives an example). Recall that $  \mathcal{L}^n_{m,p}$ denotes the set $ \bigl\{ (\ell_{1},\ldots,\ell_{m}) \in \mathbb{Z}_{\geq 0}^m : \forall i \in [\![1,m]\!] \text{ } \ell_{i} \geq 0 \text{ and } \sum_{i=1}^m \ell_{i} = n-p \bigr\}$, where $m$ and $p$ are non-negative integers.

\begin{definition} Let $ a \in \mathbb{T}_{n,k}$. The sequence of leaves of $a$, denoted by $L(a)\coloneqq(L_1(a),\ldots,L_{2k-1}(a))$, is a sequence of length $2k-1$ such that for every $i$ in $[\![1,2k-1]\!]$, the $i$-th element of the sequence, equals the number of leaves grafted in the $i$-th internal corner of $a$ if the internal corner surrounds an internal node of $R(a)$ or that number minus one if it surrounds a leaf of $R(a)$. Note that $L(a)$ is an element of $\mathcal{L}^n_{2k-1,k+\phi_0(R(a))}$.
\end{definition}

\begin{remark}\label{rqbij2}
The function $a \mapsto (R(a),L(a))$ is a bijection between $\mathbb{T}_{n,k}$ and pairs $(r,l) \in \mathbb{T}_{k} \times \mathcal{L}^n_{2k-1,k+\phi_0(r)}$. Observe that the reason why we subtract one from the number of leaves at a corner of $a$ if it is a leaf of $R(a)$ is because by the definition of $R(a)$, all its vertices are internal nodes of $a$, meaning each one must have at least one leaf in $a$.
\end{remark}
To simplify notation, we set $L \coloneqq L(T)$, $L_{n,k}\coloneqq L(T_{n,k})$.
Observe that by Remark \ref{rqbij2}, $R(a)$ and $L(a)$ completely characterize a tree $a$ with $n$ vertices and $k$ internal nodes. So studying the limit behavior of $T_{n,k}$ is equivalent to studying the joint limit behavior of  $ R_{n,k} $ and $ L_{n,k}$.

\subsection{Description via random walks}
This subsection presents a key result that will serve as the starting point for the proof of Theorem \ref{th1-intro} in Section \ref{sec5}.
\begin{table}[htbp]\caption{Table of the main notation and symbols introduced in Subsection $4.2$ and used later.}
\centering
\begin{tabular}{c c p{12cm} }
\toprule
$\mathbf{x}$ && $(x_1,\ldots,x_n)$ \\ 
$\|\mathbf{x}\|$ && $\max_{1\leq i \leq n}{x_i}$ \\
$\mathbf{x}^{\downarrow}_{i}$ && the integer obtained by sorting $\mathbf{x}$ in non-increasing order and then keeping only the $i$-th coordinate \\
$\mathbf{x}^{\downarrow}_{i:j}$ && the vector obtained by sorting $\mathbf{x}$ in non-increasing order, and then keeping only elements with index belonging to $ [\![i,j]\!] $ \\
\hline
$\ast_k$ && the tree with $1$ internal vertex and $k-1$ leaves \\
$\varnothing_a$ && the root of a tree a \\
$c^{o}_i(a)$ && the number of children of the $i$-th internal vertex, counted in lexicographic order, in $a$ \\
$\mathbf{K}(a)$ && $ (c_0(a)-1,\ldots,c_{n-1}(a)-1)$ with $a \in \mathbb{T}_{n}$ \\
\bottomrule
\end{tabular}
\end{table}

We first introduce some notation. For $a \in \mathbb{T}_{n,k}$, we denote by $\varnothing_a$ the root of a, $c^{o}_i(a)$ the number of children of the $i$-th internal vertex counted in lexicographic order in $a$ and $\mathbf{K}(a)$ the vector $(c_0(a)-1,\ldots,c_{n-1}(a)-1)$. For a vector $\mathbf{x} = (x_1,\ldots,x_p)$, we denote by $\mathbf{x}^{\downarrow} =(\mathbf{x}^{\downarrow}_{i},1 \leq i\leq p)$ the non increasing rearrangement of $\mathbf{x}$ and $\mathbf{x}^{\downarrow}_{i:j}$ the vector $(\mathbf{x}^{\downarrow}_{k}; i \leq k \leq j)$. We also note $\|\mathbf{x}\|$ the quantity $\max_{1\leq i \leq n}{x_i} = \mathbf{x}^{\downarrow}_{1}$. Then, as in the previous part, we consider a sequence $(X_i)_{i \geq 0}$ of i.i.d. random variables with distribution $\nu = \mu(\cdot +1)$ on $\mathbb{Z}_{\geq -1}$, $(Y_i)_{i\geq 0}$ a sequence $(Y_i)_{i \geq 0}$ of i.i.d. random variables such that $\mathbb{P}(Y_1=-1) = 0$ and $\mathbb{P}(Y_1=j) = \mu(j+1)/(1-\mu(0))$ for all $j \geq 0$, and two random walks $(W_i)_{i\geq0}$ and $(W'_i)_{i\geq0}$ defined by  $W_0=W'_0=0$ and $ W_i= X_1+\cdots+X_i$ and $W'_i = Y_1+\cdots+Y_i$ for all $i\geq 0$. Finally, we denote by $\ast_k$ the tree with $1$ internal vertex and $k-1$ leaves.

The following proposition essentially tells that the outdegrees of the internal nodes of $T_{n,k}$ are equal to the increments of a conditioned random walk.

\begin{proposition}\label{prop2.3} We have the following equality in distribution:
    \[\mathbf{K}^{\downarrow}_{1:k}(T_{n,k}) \quad \overset{(d)}{=} \quad \mathbf{Y}^{\downarrow}_{1:k} \text{ under } \mathbb{P}\biggl(\cdot \,\Big|\sum_{j=1}^k Y_j=n-k-1\biggr).\]
\end{proposition}

\begin{proof}
For $T \in \mathbb{T}_{n,k}$, we recall that $c_{i}(T)$ is the number of children of the $i$-th vertex is $T$ and $\mathbf{K}(T)= (c_0(T)-1,\ldots,c_{n-1}(T)-1)$ so $\mathbf{K}^{\downarrow}_{i}(T) \geq 0$ if $i \leq k$ and $\mathbf{K}^{\downarrow}_{i}(T)= -1$ otherwise. Let $ \mathbf{u} \coloneqq (u_1,\ldots,u_k) \in \mathbb{Z}_{\geq 0}^k$ be such that $u_1 \geq \cdots \geq u_k \geq 0$ and $u_1 + \cdots + u_k = n-k-1 $. We denote by $\mathbf{X}$ the vector $(X_1,\ldots,X_n)$, made of i.i.d. random variables with law $\nu$. Then, the coding of $T$ by its \L{ukasiewicz} path (Proposition \ref{propcodingluka}) shows that $\mathbb{P}(\mathbf{K}^{\downarrow}_{1:k}(T)= \mathbf{u}, T \in \mathbb{T}_{n,k})$ is equal to 
 \[\mathbb{P}\Bigl(\mathbf{X}^{\downarrow}_{1:k} = \mathbf{u} , W_n = -1 , W_j \geq 0 \,\forall i \in [\![1,n-1]\!]  ,  |\{i \text{ : } X_i =-1\}|= n-k \Bigr)\]
and applying Proposition \ref{consequencecyclelemma}, we get that this is equal to $ \mathbb{P}(\mathbf{X}^{\downarrow}_{1:k} = \mathbf{u}, |\{i \text{ : } X_i =-1\}|= n-k) / n$. Consequently, this probability is equal to
\[\frac{1}{n}\sum_{\substack{ I \subset [\![1,n]\!]\\
              |I|=n-k}}  \mathbb{P}\Bigl( \forall i \in I \, X_i = -1 , (X_j : j \notin I)^{\downarrow} = \mathbf{u} \Bigr) = \frac{1}{n}\binom{n}{n-k} \mu(0)^{n-k}\mathbb{P}\Bigl( (X_1,\ldots,X_k)^{\downarrow} = \mathbf{u} \Bigr)\]
which is equal to \[\frac{1}{n}\binom{n}{n-k} \mu(0)^{n-k} (1-\mu(0))^{k} \mathbb{P}\Bigl( (Y_1,\ldots,Y_k)^{\downarrow} = \mathbf{u} \Bigr).\]
By summing over all possible vectors $\mathbf{u}$, we also have
\begin{align*}
  \mathbb{P}\Bigl(T \in  \mathbb{T}_{n,k} \Bigr)&= \frac{1}{n}\sum_{\substack{ I \subset [\![1,n]\!] \text{ ,}\\
              |I|=n-k}}  \mathbb{P}\bigg(\forall i \in I\text{ }  X_i = -1 , \sum_{j\notin I} X_j = n-k-1, \forall j \notin I \text{ } X_j \neq -1 \biggr) \\
    & = \frac{1}{n} \binom{n}{n-k} \mu(0)^{n-k} (1-\mu(0))^{k} \text{ } \mathbb{P}\biggl( \sum_{j=1}^k Y_j = n-k-1  \biggr).
\end{align*}
Hence, $\mathbb{P}(\mathbf{K}^{\downarrow}_{1:k}(T_{n,k}) = \mathbf{u} ) = \mathbb{P}(\mathbf{K}^{\downarrow}_{1:k}(T) = \mathbf{u} \,|\, T \in  \mathbb{T}_{n,k} )$ is equal to 
\[\frac{\binom{n}{n-k} \mu(0)^{n-k} (1-\mu(0))^{k} \text{ } \mathbb{P}\bigl( \mathbf{Y}^{\downarrow}_{1:k} = \mathbf{u}, \sum_{j=1}^k Y_j = n-k-1 \bigr)}{\binom{n}{n-k} \mu(0)^{n-k} (1-\mu(0))^{k} \text{ } \mathbb{P}\bigl( \sum_{j=1}^k Y_j = n-k-1 \bigr)}\]
which is equal to $\mathbb{P}( \mathbf{Y}^{\downarrow}_{1:k} = \mathbf{u} \mid \sum_{j=1}^k Y_j = n-k-1)$.
\end{proof}

\subsection{Description via generating functions}
This subsection introduces a set of tools and a key result that will serve as a common starting point for the proofs of Theorems \ref{th3.2}, \ref{th_cvg_transfert_case-intro} and \ref{thPois-intro} respectively detailed in Sections \ref{sec6}, \ref{sec7} and \ref{sec8}.
\begin{table}[htbp]\caption{Table of the main notation and symbols introduced in Section $4.3$ and used later.}
\centering
\begin{tabular}{c c p{12cm} }
\toprule
$I(a)$ && the set of internal nodes in $a \in \mathbb{T}_{k}$\\
$\Tilde{c}_u(a)$ && $c_u(a)$ if $u$ is not a leaf of $a$, $1$ otherwise\\
$\zeta_k(a)$ && $(\prod_{u \in I(a)} c_u(b)!)^{-1}$, $a \in \mathbb{T}_k$ \\
\hline
$\widetilde{F}_{\mu}$ && the generating function of $\mu(\cdot+1)$  \\
$G^a_{\mu} (z)$ && $\prod_{u \notin I(a)}\widetilde{F}_{\mu} (z)\prod_{u \in I(a)} F_{\mu}^{(c_u(a))}(z)$\\
\bottomrule
\end{tabular}
\end{table}

For $a \in \mathbb{T}_k$, we denote by $I(a)$ the set of internal nodes in $a$, and for a vertex $u$ of $a$ we define $\Tilde{c}_u(a)$ as follows: 
\[\Tilde{c}_u(a) \coloneqq \left\{
    \begin{array}{ll}
        c_u(a) & \mbox{if } u \mbox{ is not a leaf of } a\\
        1 & \mbox{otherwise.} 
    \end{array}
\right.\]
We recall that for $i \in \mathbb{Z}_{\geq 0}$, $\phi_i(a)$ is the number of vertices in $a$ with $i$ children. We denote by $F_{\mu}^{(j)}$ the $j$-th derivative of $F_{\mu}$ and $\widetilde{F}_{\mu}$ the generating function of $\mu(\cdot+1)$. Finally, we define $G^a_{\mu}$ as follows \[G^a_{\mu} (z)\coloneqq \prod_{u \notin I(a)}\widetilde{F}_{\mu} (z)\prod_{u \in I(a)} F_{\mu}^{(c_u(a))}(z).\]

The following lemma expresses the reduced tree in terms of the coefficients of $G^a_{\mu}$. 

\begin{lemma}\label{lemformuleRnk}
For all $a \in \mathbb{T}_k $,
    \[\mathbb{P}(R_{n,k} = a) = \frac{\zeta_k(a)[z^{n-k-\phi_0(a)}] G^a_{\mu} (z)}{\sum_{b \in \mathbb{T}_k} \zeta_k(b)[z^{n-k-\phi_0(b)}] G^b_{\mu} (z)},\] where $\zeta_k(b) \coloneqq (\prod_{u \in I(b)} c_u(b)!)^{-1} $ for all $b \in \mathbb{T}_k $.

\end{lemma}
\begin{proof}
Fix $a \in \mathbb{T}_k $. On the one hand, we have
\[ \widetilde{F}_{\mu}(z) = \sum_{i\geq 0} \mu(i+1)z^i  \quad \text{and} \quad F_{\mu}^{(c_u(a))}(z) = \sum_{i \geq c_u(a) } \binom{i}{c_u(a)} (c_u(a))! \, \mu(i) z^{i - c_u(a)},\]
so we get
\begin{align*}
G^a_{\mu}(z) & = \biggl(\prod_{u \notin I(a)} \sum_{i\geq 0} \mu(i+1)  z^i \biggr)\biggl(\prod_{u \in I(a)}  \sum_{i \geq 0 } \binom{c_u(a) + i}{c_u(a)} (c_u(a))!  \mu(i + c_u(a)) z^i \biggr) \\
& = \sum_{\substack{(i_u)_{1 \leq u \leq k} \in \mathbb{Z}^k_{\geq 0} }} \biggl( \prod_{u \notin I(a)} \mu(i_u +1)\prod_{u \in I(a)}\binom{c_u(a) + i}{c_u(a)} (c_u(a))! \mu(i + c_u(a)) \biggr) z^{i_1+\cdots+i_k}.
\end{align*}
Then, we deduce that,
\begin{equation} \label{expressionG}
    \left[ z^{n-k-\phi_0(a)} \right] G^a_{\mu}(z) = \sum_{\substack{(i_u)_{1 \leq u \leq k}\\ \sum i_u = n-k-\phi_0(a)} } \prod_{u \notin I(a)} \mu(i_u +1)\prod_{u \in I(a)}\binom{c_u(a) + i}{c_u(a)} (c_u(a))! \, \mu(i + c_u(a)).
\end{equation}
On the other hand, we have
\[\mathbb{P}(R_{n,k} = a) = \frac{\mathbb{P}(R = a , T \in \mathbb{T}_{n,k} )}{\sum_{b \in \mathbb{T}_k} \mathbb{P}(R = b , T \in \mathbb{T}_{n,k} )}. \]
Let us denote by $C_a$ the set of indices of corners in $a$ and $C_a(u)$ the set of indices of corners around the vertex $u$ in $a$. By the definition of $R$, all its vertices are internal nodes of $T$, meaning each must have at least one leaf attached to it in $T$. Consequently, to construct $T$ from $R=a$, we first attach a leaf to each leaf of $R$, and then we add the remaining $n-k-\phi_0(a)$ leaves arbitrarily at the corners of $R$. Hence, by summing over all possible positions of the $n-k-\phi_0(a)$ leaves in the corners of $a$, we get that $ \mathbb{P}(R = a , T \in \mathbb{T}_{n,k} )$ is equal to 
\[ \sum_{ \substack{(q_c)_{c \in C_a}\\ \sum q_c = n-k-\phi_0(a)}} \mu(0)^{n-k} \prod_{u \in a} \mu\biggl(\Tilde{c}_u(a) \text{ } +   \sum_{c \in C_a(u)} q_c \biggr) = \mu(0)^{n-k} \sum_{ \substack{(i_u)_{1 \leq u \leq k}\text{ }\\ \sum i_u = n-k-\phi_0(a)}}  \prod_{u = 1}^k \binom{c_u(a) + i_u}{c_u(a)} \mu(\Tilde{c}_u(a) + i_u)\]
which is equal to \[\zeta_k(a)\mu(0)^{n-k} \sum_{ \substack{(i_u)_{1 \leq u \leq k}\text{ }\\ \sum i_u = n-k-\phi_0(a)}}  \prod_{u \notin I(a)} \mu( i_u +1) \prod_{u \in I(a)} \binom{c_u(a) + i_u}{c_u(a)} (c_u(a))! \textbf{ } \mu( c_u(a) + i_u).\]
Therefore, by Equation \eqref{expressionG} we obtain that\[\mathbb{P}(R = a , T \in \mathbb{T}_{n,k} )  =  \zeta_k(a)\mu(0)^{n-k}  \left[ z^{n-k-\phi_0(a)} \right] G^a_{\mu}(z), \]
and the lemma follows.
\end{proof}

\section{Local setting: Proof of Theorem \ref{th1-intro}}
\label{sec5}
From now on, we aim to study the limiting behavior of a $\mu$-BGW tree conditioned to have $n$ vertices and a fixed number $k \in \mathbb{Z}_{\geq 0}$ of internal nodes, as $n$ tends to infinity. In contrast to Section \ref{sec3}, no universal result holds for arbitrary offspring distributions. Instead, the asymptotic behavior of the tree depends strongly on the specific properties of $\mu$. We therefore distinguish several regimes, each corresponding to a different class of offspring distributions. Nevertheless, it is interesting to note that in all the cases we cover, the reduced tree will converge in distribution to a simply generated tree.

We recall the definition of a slowly varying function.
\begin{definition}\label{defslowlyvarfunc}
    A function $\ell :\mathbb{R}_+\rightarrow\mathbb{R}^*_+$ is slowly varying if for all $a >0$ \[\frac{\ell(ax)}{\ell(x)} \xrightarrow[x \rightarrow +\infty]{} 1.\]
    We denote by $\mathcal{R}_0$ the set of all slowly varying functions.
\end{definition}
In this part, we study the behavior of a $\mu$-BGW tree under the following local assumption on $\mu$, denoted by \eqref{hloc}:
\begin{equation}\label{hloc}
    ``\text{There exists } \ell \in \mathcal{R}_0 \text{ and } \beta > 1 \text{ such that for all } i \geq 0, \, \mu(i)=\ell(i)/i^{1+\beta}."
    \tag{$\mathcal{H}_{loc}$}
\end{equation} 
We aim to prove Theorem \ref{th1-intro} which states that $T_{n,k}$ is `star-liked' with high probability. To this end, we will use the encoding of a tree by its \L ukasiewicz path and the key result Proposition \ref{prop2.3}.

First, we show that if we remove the vertex with the largest degree, the degree of the other vertices behave asymptotically like i.i.d. random variables with the same distribution as $Y_1$. This will imply that for large $n$, the vertex with the largest degree will have a size of order $n$, while the degrees of the remaining vertices will remain bounded. 
\begin{proposition} \label{prop2.5}
Under the assumption \eqref{hloc} we have
\[d_{TV}\bigl( \mathbf{K}^{\downarrow}_{2:k}(T_{n,k}), \mathbf{Y}^{\downarrow}_{1:k-1} \bigr)\xrightarrow[n \rightarrow +\infty]{} 0.\]
Consequently, for all $ \epsilon > 0$, there exists $ M > 0$ such that for n large enough, \[ \mathbb{P}\Bigl(|\mathbf{K}^{\downarrow}_{1}(T_{n,k})- n| < M  \text{ and } \|\mathbf{K}^{\downarrow}_{2:k}(T_{n,k})\| < M \Bigr) \geq 1-\epsilon. \]
\end{proposition}

\begin{proof}
      Applying Equation (2.9) of \cite{AL11} (see also \cite{F07condensation}), we get that the total variation distance between $\mathbf{Y}_{2:k}^{\downarrow}$ under $\mathbb{P}(\cdot \,|\,\sum_{j=1}^k Y_j=n-k-1)$ and $\mathbf{Y}^{\downarrow}_{1:k-1}$, unconditioned, goes to $0$ when $n$ tends to infinity. Then we get the convergence in total variation distance by applying Proposition \ref{prop2.3}. Moreover, we have $\mathbb{P}\bigl( \exists M \in \mathbb{Z}_{\geq 0} \text{ st } \|\mathbf{Y}^{\downarrow}_{1:k-1}\| < M \bigr) = 1 $ because $\|\mathbf{Y}^{\downarrow}_{1:k-1}\|$ is a finite random variable.
      So, for all $\epsilon > 0$ there exists an integer $M$ such that for $n$ large enough, $\mathbb{P} \bigl(\|\mathbf{K}^{\downarrow}_{2:k}(T_{n,k})\| < M \bigr) \geq 1-\epsilon.$
Moreover, on the event $\{ \|\mathbf{K}^{\downarrow}_{2:k}(T_{n,k})\| < M \}$, we have 
\[K^{\downarrow}_{1}(T_{n,k}) = n -(k+1) - \sum_{i=2}^k K^{\downarrow}_{i}(T_{n,k}) \geq n - (k+1) - M(k-1). \]
So $| K^{\downarrow}_{1}(T_{n,k}) - n | \leq M(k-1) + k+1$. Thus, taking $ M' \coloneqq M(k-1) +k+1 > M $ we get
\[ \mathbb{P}\Bigl(\|\mathbf{K}^{\downarrow}_{2:k}(T_{n,k})\| < M' \text{ and }|\mathbf{K}^{\downarrow}_{1}(T_{n,k})-n| < M' \Bigr) \geq 1-\epsilon . \]This concludes the proof.
\end{proof}
Now, we prove that with high probability the root is the vertex with the largest degree and that $R_{n,k}$ converges in distribution to $\ast_k$.
\begin{proposition}\label{prop2.6} Under the assumption \eqref{hloc} we have the following convergence:
    \[\mathbb{P}\Bigl( R_{n,k} = \ast_k \text{ and the root has maximal degree in } T_{n,k}  \Bigr) \xrightarrow[n \rightarrow +\infty]{} 1. \]
\end{proposition}

\begin{proof}
For all $i \in [\![0,k-1]\!] $ and $t \in \mathbb{T}_{n,k}$ we let $I_i(t)$ be the event defined by \[I_i(t) \coloneqq \left\{ \text{the } \textit{i}\text{-th internal node has maximal degree in } t \right\}.\] For $M > 0$ and $t \in \mathbb{T}_{n,k}$, set \[A_{M,n}(t) \coloneqq \left\{ |\mathbf{K}^{\downarrow}_{1}(t)- n| < M  \text{ and } \|\mathbf{K}^{\downarrow}_{2:k}(t)\| < M \right\}.\]
Observe that for sufficiently large $n$, the event $A_{M,n}$ is equivalent to the existence of a unique vertex in $t$, denoted by  $u^{\star}$, such that $c_{u^{\star}}(t) > n-M$.
Using Proposition \ref{prop2.5}, fix $\epsilon > 0$ and $M > 0$ such that for $n$ large enough, $\mathbb{P}\bigl(A_{M,n} (T_{n,k}) \bigr) \geq 1 - \epsilon$. Then, it is enough to prove the convergence
\[\mathbb{P}\bigl( R_{n,k} = \ast_k, I_0(T_{n,k})  \big| A_{M,n}(T_{n,k}) \bigr) = \frac{\mathbb{P}\bigl( T \in \mathbb{T}_{n,k}, R = \ast_k, A_{M,n}(T), I_0(T) \bigr)}{\mathbb{P}\bigl( T \in \mathbb{T}_{n,k}, A_{M,n}(T) \bigr)} \xrightarrow[n \rightarrow +\infty]{} 1. \]
For $r \in \mathbb{T}_{k}, u$ an internal vertex of $r$ and $C \geq 1,$ we define 
\[\mathcal{A}^{r,u}_{C,n}\coloneqq \left\{ (a_v)_{1 \leq v \leq k} \in \mathbb{Z}_{\geq 0}^k : a_v \geq 1 \text{ if } v \text{ is a leaf in } r \text{ and } \sum_{v \in r \setminus \{u\}}a_v= n-C-1 \right\}.\]
Thus, we have
\[\mathbb{P}( T \in \mathbb{T}_{n,k}, R=r, c_u(T)=C )=\binom{C}{c_u(r)}\mu(0)^{n-k}\mu(C)  \sum_{(a_v) \in \mathcal{A}^{r,u}_{C,n}} \prod_{v \neq u} \binom{a_v}{c_v(r)} \mu(a_v).\]
Observe that $ (C-c_u(r))^{c_u(r)} \leq \binom{C}{c_u(r)} \leq C^{c_u(r)}$ and $1 \leq \binom{a_v}{c_v(r)} \leq a_v^{c_v(r)}$, so if $n \geq C > n-M$ we get $(n-M-c_u(r))^{c_u(r)} \leq \binom{C}{c_u(r)} \leq n^{c_u(r)}$ and $1 \leq \binom{a_v}{c_v(r)} \leq M^{k}$. Consequently, if $r=\ast_k$ and $u = \varnothing_{r}$, then we have $c_u(r)=k-1$ and $\mathbb{P}(T \in \mathbb{T}_{n,k}, R=r, c_u(T) > n-M)$ is at least
\[\mu(0)^{n-k}(n-M-k+1)^{k-1}\mu(C)\sum_{C > n-M}\sum_{(a_v)\in \mathcal{A}^{r,u}_{C,n}}\prod_{v \neq u} \binom{a_v}{c_v(r)} \mu(a_v), \]
whereas if either $r\neq \ast_k$ or $r=\ast_k$ and $u \neq \varnothing_r$, then $c_u(r) \leq k-2$ and $\mathbb{P}\bigl(T \in \mathbb{T}_{n,k}, R=r, c_u(T) > n-M)$ is at most \[\mu(0)^{n-k}n^{k-2}M^k\mu(C)\sum_{C > n-M}\sum_{(a_v) \in \mathcal{A}^{r,u}_{C,n}}\prod_{v \neq u} \binom{a_v}{c_v(r)} \mu(a_v).\]
Therefore, we have \[\frac{\mathbb{P}\bigl(T \in \mathbb{T}_{n,k},A_{M,n}(T), c_{u^{\star}}(R) \leq k-2 \bigr) }{\mathbb{P}(T \in \mathbb{T}_{n,k}, A_{M,n}(T), R=\ast_k, I_0(T) \bigr)} \xrightarrow[n \rightarrow +\infty]{} 0.\]
Noting that under the event $\{T \in \mathbb{T}_{n,k},A_{M,n}(T)\}$, we have either $ c_{u^{\star}}(R) \leq k-2 $ or $R=\ast_k$ and $ I_0(T)$, and using the fact that  $b/(a+b) \rightarrow 1$ when $a/b \rightarrow 0$, we obtain the result.
\end{proof}

We can now proceed to the proof of the theorem.
\begin{proof}[\textit{Proof of Theorem }\ref{th1-intro}]
First, for all $a \in \mathbb{T}_{n,k}$ such that $R(a) \neq \ast_k$, $\mathbb{P}(T_{n,k}=a)$ tends to $0$ by Proposition \ref{prop2.6} and $\mathbb{P}(D_{n,k}=a)$ tends to $0$ by construction (as $\mathbb{P}(G_{n,k}) \xrightarrow[n \rightarrow +\infty]{} 1$). Then, we have that $d_{TV}(T_{n,k},D_{n,k})$ is equal to
\[\frac{1}{2}\sup_{a \in \mathbb{T}_{n,k}}\bigl| \mathbb{P}(T_{n,k}=a) - \mathbb{P}(D_{n,k}=a) \bigr| = \frac{1}{2}\sup_{\substack{a \in \mathbb{T}_{n,k}\\ R(a) = \ast_k}}\bigl| \mathbb{P}(T_{n,k}=a) - \mathbb{P}(D_{n,k}=a) \bigr| + o(1).\]
Now as the reduced tree $R(a)$ is fixed equal to $\ast_k$, according to Remark \ref{rqbij2} we just have to look at the sequence of leaves. Let $L(D_{n,k})$ be the list of leaves of $D_{n,k}$ and define the bijection $f : \mathbb{Z}_{\geq 0}^{2k-1} \rightarrow \{ a:R(a) = \ast_k\}$ as shown in Figure \ref{Fig4}. Then, $ 2d_{TV}(T_{n,k},D_{n,k})$ is equal up to $o(1)$ to
\begin{figure}[h] 
    \centering
    \includegraphics[scale = 1]{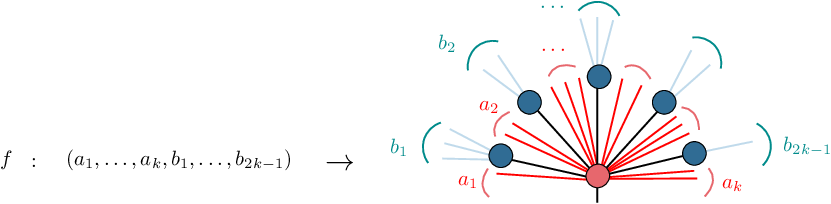}
    \caption{Definition of $f$.}
    \label{Fig4}
\end{figure}
\begin{align*}
  \sup_{\substack{a \in \mathbb{T}_{n,k}\\ R(a) = \ast_k}}\Bigl| \mathbb{P}\bigl(f(L(T_{n,k}))=a\bigr) - \mathbb{P}\bigl(f(L(D_{n,k}))=a \bigr) \Bigr| & = \sup_{\substack{a \in \mathbb{T}_{n,k}\\ R(a) = \ast_k}}\Bigl| \mathbb{P}\bigl(L(T_{n,k})=f^{-1}(a)\bigr) - \mathbb{P}\bigl(L(D_{n,k})=f^{-1}(a)\bigr) \Bigr| \\ 
  & = \sup_{s \in \mathbb{Z}_{\geq 0}^{2k-1}}\Bigl| \mathbb{P}\bigl(L(T_{n,k})=s\bigr) - \mathbb{P}\bigl(L(D_{n,k})=s \bigr) \Bigr|.
\end{align*}
Moreover, conditionally given the event ``$R_{n,k} = \ast_k \text{ and the root has maximal degree in } T_{n,k}$'', the number of children of the $k-1$ leaves of $R_{n,k}$ in $T_{n,k} $ are exchangeable by definition of the law of a $\mu$-BGW tree. So, $(L_{k+1}(T_{n,k}),\ldots,L_{2k-1}(T_{n,k}))$ has the law of an independent uniform permutation of $\mathbf{K}^{\downarrow}_{2:k}(T_{n,k})$.
Hence, applying Propositions \ref{prop2.5} and \ref{prop2.6}, we deduce that \[d_{TV}\bigl((L_{k+1}(T_{n,k}),\ldots,L_{2k-1}(T_{n,k})),(L_{k+1}(D_{n,k}),L_{2k-1}(D_{n,k}))\bigr)\xrightarrow[n \rightarrow +\infty]{} 0.\]
Finally, conditionally given $R_{n,k}= \ast_k$ and $(L_{k+1}(T_{n,k}),\ldots,L_{2k-1}(T_{n,k}))$, the vector
$(L_1(T_{n,k}),\ldots,\\L_{k}(T_{n,k}))$ is a composition of $n-k-\sum_{i=k}^{2k-1} L_{i}(T_{n,k})$ into $k$ parts sampled uniformly at random, so is $(L_{1}(D_{n,k}),\ldots,L_{k}(D_{n,k}))$ on $G_{n,k}$, and we get the desired result.
\end{proof}

\section{Tail setting: Proof of Theorem \ref{th3.2}} \label{sec6}
Up to this point, the main tool for proving our results has been the description of a BGW tree through its \L ukasiewicz path. From now on, we will completely shift our approach and rely instead on techniques from analytic combinatorics. In fact, the key result for describing the behavior of the reduced tree of a BGW tree conditioned to have $n$ vertices and $k$ internal nodes will be Lemma \ref{lemformuleRnk}, which expresses it in terms of the coefficients of a generating function and its derivatives. Our goal is therefore be to analyze these coefficient extractions.
\begin{table}[htbp]\caption{Table of the main notation and symbols introduced in Section \ref{sec6} and used later.}
\centering
\begin{tabular}{c c p{12cm} }
\toprule
$\mathcal{H}$ && ``There exists a positive real number $u_0$ and an angle $0 < \omega < \pi/2 $ such that $L(u) \neq 0$ and is analytic in the domain $ \{ u :  \omega -\pi \leq Arg(u-u_0) \leq \pi - \omega\}$ and satisfies for $\theta \in (\omega-\pi,\pi - \omega), u \geq u_0$:
\[\lvert \frac{L(ue^{i\theta})}{L(u)} - 1 \rvert < \epsilon(u) \text{ and } \lvert \frac{L(u\log^2(u))}{L(u)} - 1 \rvert < \epsilon(u). \text{''} \]\\
$\mathcal{L}$ && the set of slowly varying functions which satisfies $\mathcal{H}$\\
\hline
$C_a$ && indices of corners in the tree $a$ \\
$C_a(u)$ && indices of corners around the vertex $u$ in $a$  \\
\hline 
$m$ && a real number greater than $-1$ \\
$\alpha$ && a real number in $ (1,2) $ \\
$\hat{z}$ && $(1-z)^{-1}$ where $0<z<1$ \\
$\ell$ && a function in  $\mathcal{L}$ \\
$S(u)$ && $ \ell(u)u^{-\alpha}$ where $\forall u > 0$ \\
\bottomrule
\end{tabular}
\end{table}

In this section, we study the behavior of a $\mu$-BGW tree under a ``tail'' assumption on $\mu$. Let us define what this means. We denote by $\mathcal{L}$ the set of functions $L: \mathbb{C}\rightarrow\mathbb{C}$ satisfying the following assumption $\mathcal{H}$ and such that the restriction on $\mathbb{R}$ is slowly varying. Assumption $\mathcal{H}$: ``There exists a positive real number $u_0$, an angle $0 < \omega < \pi/2$ and a function $ \epsilon : \mathbb{R}_+ \rightarrow \mathbb{R}$  with $\epsilon(x) \underset{x \to \infty}{\rightarrow} 0$  such that $ L(u) \neq 0 $ and is analytic in the domain $ \{ u:\omega-\pi \leq Arg(u-u_0) \leq \pi - \omega\} $ and satisfies for $ \theta \in (\omega-\pi,\pi -\omega), u \geq u_0$:\[ \lvert \frac{L(ue^{i\theta})}{L(u)} - 1 \rvert < \epsilon(u) \text{ and } \lvert \frac{L(u\log^2(u))}{L(u)} - 1 \rvert < \epsilon(u) \text{''}.\]
We introduce this technical assumption $\mathcal{H}$ with the aim of applying Theorem $5$ from \cite{FO90} later in the proof of Theorem \ref{th3.2}. This theorem connects the behavior of a generating function near its critical point with the asymptotic behavior of its coefficients.
Let $m$ will be a real number greater than $-1$ and $\alpha$ a real number in $ (1,2)$. Let $\mu$ be a probability distribution on $\mathbb{Z}_{\geq 0}$ with mean equal to $1+m$ and generating function $F_{\mu}$ given by Equation \eqref{defF}: \begin{equation*}
 F_{\mu}(z) = z -mz +mz^2 + \ell\Bigl(\frac{1}{1-z}\Bigr)(1-z)^\alpha,   
\end{equation*} where $\ell$ is in $\mathcal{L}$. In particular, taking $n=1$ in Theorem 8.1.6 of \cite{BGT97}, one can check that $F_{\mu}$ satisfies (8.1.9) there, which is equivalent to (8.1.11b), and that implies 
\begin{equation}\label{propmu}
\mu([n,\infty[) \underset{n \to \infty}{\sim} \frac{\ell(n)}{-\Gamma(1-\alpha) n^\alpha}.
\end{equation}

\begin{remark} \label{rem1}
One can check that if $\ell \in \mathcal{L}$, then for all $p \in \mathbb{Z}_{\geq 0}, \ell^p \in \mathcal{L}$.
\end{remark}

Recall that $c^{o}_i(a)$ is the number of children of the $i$-th internal vertex counted in lexicographic order in $a$. The aim of this section is to prove Theorem \ref{th3.2}, illustrated in Figure \ref{Fig5}. 

Firstly, the third point of Theorem \ref{th3.2} follows directly by applying the following lemma (which is a straightforward consequence of Lemma \ref{thcvgdir}) with $P_n = n-c^o_0(T_{n,k})$, which is a $o(n)$ as established in the second point of Theorem \ref{th3.2}.
\begin{lemma}\label{corcvgdir}
 Let $m \geq 2$ and $(P_n)_{n\geq 0}$ be a sequence of integer-valued random variables such that \[ \frac{P_n}{n} \xrightarrow[n \rightarrow +\infty]{(\mathbb{P})} 0.\]
 Conditionally given $(P_n)_{n\geq 0}$, let $(U_{n,1},\ldots,U_{n,m})$ be uniform on $\bigl\{ (\ell_{1},\ldots,\ell_{m}) :  \forall i \in [\![1,m]\!] \text{ } \ell_{i} \geq 0 \text{ and } \sum_{i=1}^m \ell_{i} = n-P_n \bigr\}$. Then we have \[\frac{(U_1^{n},\ldots , U_m^{n})}{n} \xrightarrow[n \rightarrow \infty]{(d)} \mathbf{Dir}(1,\ldots,1).\]
\end{lemma}
In the next two subsections, we prove the first two parts of Theorem~\ref{th3.2}.

\subsection{The reduced tree converges in distribution to a star}
We recall from Subsection $4.3$ that for $a \in \mathbb{T}_k$, we denote by $I(a)$ the set of internal nodes in $a$. For a vertex $u$ of $a$, the quantity $\Tilde{c}_u(a)$ is defined as: 
\[\Tilde{c}_u(a) \coloneqq \left\{
    \begin{array}{ll}
        c_u(a) & \mbox{if } u \mbox{ is not a leaf of } a\\
        1 & \mbox{otherwise.} 
    \end{array}
\right.\]
Moreover, for each $i \in \mathbb{Z}_{\geq 0}$, $\phi_i(a)$ denotes the number of vertices in $a$ with $i$ children. We also recall that $\widetilde{F}_{\mu}$ is the generating function of $\mu(\cdot+1)$ and that $G^a_{\mu}$ is defined by \[G^a_{\mu} (z)\coloneqq \prod_{u \notin I(a)}\widetilde{F}_{\mu} (z)\prod_{u \in I(a)} F_{\mu}^{(c_u(a))}(z).\]

The aim of this part is to establish the first statement of Theorem \ref{th3.2}. By Lemma \ref{lemformuleRnk}, it is sufficient to analyze the behavior of $[z^{n-k-\phi_0(a)}]G^a_{\mu}(z)$ as $n$ tends to infinity to determine the asymptotics of $\mathbb{P}(R_{n,k} = a)$. 

For $k \in \{0,1,2\}$, $|\mathbb{T}_{k}|=1$ so the result is clear for these $k$. From now on, we assume $k \geq 3$. To simplify notation, we set $S(u) = \ell(u)u^{-\alpha}$ for all $u>0$ and $\hat{z} = (1-z)^{-1}$ for $0<z<1$. 

The following lemma establishes a stability property of derivatives for $S$, which will be useful to get the asymptotic of $G^a_{\mu}(z)$ since, according to Equation \eqref{defF}, $F_{\mu}(z) =  z -mz +mz^2 + S(\hat{z}) $.

\begin{lemma}\label{lemS}
For all $p\in \mathbb{Z}_{\geq 0}$, we have the following convergence:
\[ \frac{u^pS^{(p)}(u)}{S(u)} \xrightarrow[u \rightarrow \infty]{} \frac{\Gamma(-\alpha + 1)}{\Gamma(-\alpha + 1 - p)}. \]

\end{lemma}

\begin{proof}
For all $u > 0$, set $g(u) \coloneqq S(1/u) = u^\alpha\ell(1/u)$ so that we have $g(z) = F_{\mu}(1-z) - P(1-z)$, where $P$ is the polynomial defined by $P(x)=mx^2+(1-m)x$. It readily follows that $g$ is infinitely differentiable and has monotone derivatives on a neighborhood of $0$. Consequently, applying Theorem 2 of \cite{Lamp}, we get
\[x\frac{g'(x)}{g(x)} \xrightarrow[x \rightarrow 0^{+}]{} \alpha. \]
By setting $h_1(x)=g'(1/x)$, we deduce that for all $\lambda > 0$
\[\frac{h_1(\lambda x)}{h_1(x)} \underset{x \to +\infty}{\sim} \lambda^{1-\alpha}\frac{\ell(\lambda x)}{\ell(x)} \xrightarrow[x \rightarrow+\infty]{} \lambda^{1-\alpha}, \]
because $\ell$ is slowly varying. Thus, $h_1$ is regularly varying and can be written $h_1(x)=x^{1-\alpha}\ell_1(x)$ with $\ell_1$ a slowly varying function, so $g'(x)=x^{\alpha-1}\ell_1(1/x)$. Hence, we can again apply Theorem 2 of \cite{Lamp} to $g'$ to get that \[x^2\frac{g^{(2)}(x)}{g(x)} \xrightarrow[x \rightarrow 0^{+}]{} \alpha(\alpha-1). \]
We similarly prove by induction that for all positive integers $p$, we have  \begin{equation} \label{cvg_g}
    x^p\frac{g^{(p)}(x)}{g(x)} \xrightarrow[x \rightarrow 0^{+}]{} \frac{\Gamma(\alpha+1)}{\Gamma(\alpha +1-p)}.
\end{equation}
Now, using the Faà di Bruno's derivation formula, we get that for all positive integers $p$,
\begin{align*}
    S^{(p)}(u) &= \sum_{\substack{m_1,\ldots,m_p \in \mathbb{Z}_{\geq 0}\\
             m_1+2m_2+\cdots+pm_p=p}} \frac{p!}{m_1! \cdots m_p!}g^{(m_1+\cdots+m_p)}(1/u)\prod_{i=1}^p \frac{(-1)^{im_i}}{u^{(i+1)m_i}} \\
             & = \sum_{\substack{m_1,\ldots,m_p \in \mathbb{Z}_{\geq 0}\\
             m_1+2m_2+\cdots+pm_p=p}} \frac{(-1)^p}{u^p} \frac{p!}{m_1! \cdots m_p!} \frac{g^{(m_1+\cdots+m_p)}(1/u)}{u^{m_1+\cdots+m_p}}.
\end{align*}
So, by Equation \eqref{cvg_g}, we have
\begin{align*}
  (-1)^p u^p \frac{S^{(p)}(u)}{S(u)} & = \sum_{\substack{m_1,\ldots,m_p \in \mathbb{Z}_{\geq 0}\\
             m_1+2m_2+\cdots+pm_p=p}} \frac{p!}{m_1! \cdots m_p!}\frac{1}{u^{(m_1+\cdots+m_p)}} \frac{g^{(m_1+\cdots+m_p)}(1/u)}{g(1/u)} \\
             & \underset{u \to \infty}{\sim} \sum_{\substack{m_1,\ldots,m_p \in \mathbb{Z}_{\geq 0}\\
             m_1+2m_2+\cdots+pm_p=p}} \frac{p!}{m_1! \cdots m_p!} \frac{\Gamma(\alpha+1)}{\Gamma(\alpha +1-\sum_im_i)}.
\end{align*}
Consequently, as $(-1)^p \frac{\Gamma(-\alpha + 1)}{\Gamma(-\alpha + 1 - p)} = \frac{\Gamma(\alpha+p)}{\Gamma(\alpha)}$, the claim follows if we prove the following equality
\begin{equation}\label{combinatorial_equality}
 \sum_{\substack{m_1,\ldots,m_p \in \mathbb{Z}_{\geq 0}\\ m_1+2m_2+\cdots+pm_p=p}} \frac{p!}{m_1! \cdots m_p!} \frac{\Gamma(\alpha+1)}{\Gamma(\alpha +1-\sum_im_i)} = \frac{\Gamma(\alpha+p)}{\Gamma(\alpha)}.
 \end{equation}
Fix $N$ and $p$ positive integers such that $p <N$ and let us count in two different ways the number of arrangements of $p$ balls into $N$ urns. On one hand, this is equal to $\binom{N -1 +p}{p}$, on the other hand we can count this in that way: We chose $m_1$ urns which contain one ball among the $N$ urns. Then we chose $m_2$ urns which contain two balls among $N - m_1$ urns. We do that $p$ times, and so we get that the number of arrangements of $p$ balls into $N$ urns is also equal to \[\sum_{\substack{m_1,\ldots,m_p \in \mathbb{Z}_{\geq 0}\\ m_1+2m_2+\cdots+pm_p=p}}\binom{N}{m_1,\ldots,m_p,N-\sum_im_i}.\] 
Therefore, we have the following equality for all integers $N,p$ such that $p <N$, \[\sum_{\substack{m_1,\ldots,m_p \in \mathbb{Z}_{\geq 0}\\ m_1+2m_2+\cdots+pm_p=p}}\binom{N}{m_1,\ldots,m_p,N-\sum_im_i} = \binom{N -1 +p}{p},\]
which is equivalent to \[\sum_{\substack{m_1,\ldots,m_p \in \mathbb{Z}_{\geq 0}\\ m_1+2m_2+\cdots+pm_p=p}}\frac{p!}{m_1! \cdots m_p!} \frac{\Gamma(N+1)}{\Gamma(N +1-\sum_im_i)} = \frac{\Gamma(N+p)}{\Gamma(N)}.\]
So, we have equality between two polynomials with degrees less than $p$ for an infinite numbers of $N$, so they are equal everywhere. In particular, we get \eqref{combinatorial_equality}.
\end{proof}

Now, we can express the asymptotics of the $p$-th derivative of $S(\hat{z})$ in terms of $S(\hat{z})$ as $z \rightarrow1$. 
\begin{lemma} \label{lem3.7}
    For all $p \in \mathbb{Z}_{>0}$, there exists a constant $C_{\alpha,p}$ such that 
    \[(S(\hat{z}))^{(p)}\underset{z \to 1}{\sim} C_{\alpha,p} \hat{z}^{p+1}S'(\hat{z}) \underset{z \to 1}{\sim}  -\alpha C_{\alpha,p} \hat{z}^{p}S(\hat{z}),\]
where $(S(\hat{z}))^{(p)}$ is the $p$-th derivative of $S((1-z)^{-1})$ with respect to $z$ and $S^{(p)}(\hat{z})$ is the $p$-th derivative of $S$ evaluated at $\hat{z}$. 
   
\end{lemma}

\begin{proof}
Let us prove the first asymptotic equivalent. We define for all $m \geq 1$
 \[D_m(z) \coloneqq \biggl(\frac{1}{(1-z)^2}\biggr)^{(m-1)}S'((1-z)^{-1})) = \frac{m!}{(1-z)^{m+1}}S'((1-z)^{-1}))= m!\hat{z}^{m+1}S'(\hat{z}).\] We have
\[\textbf{ } (S(\hat{z}))'  = \hat{z}^2 S'(\hat{z}) = D_1(z) \]
and then
\[\textbf{ } (S(\hat{z}))^{(2)} =  \hat{z}^4 S^{(2)}(\hat{z}) + D_2(z).\]
By Lemma \ref{lemS}, we get \[\frac{(S(\hat{z}))^{(2)}}{D_2(z)}  = \frac{\hat{z}}{2}\frac{S^{(2)}(\hat{z})}{S'(\hat{z})} + 1 \xrightarrow[z \rightarrow 1]{} \frac{(\alpha - 1)}{2} + 1,\] 
so \[(S(\hat{z}))^{(2)}\underset{z \to 1}{\sim} \biggl(\frac{(\alpha - 1)}{2} + 1 \biggr) \textbf{ } D_2(z).\]
By induction on $p$, we get that for all $p \geq 1 $, there exist constants $c_1=p!,c_2,\ldots,c_p$ such that $(S(\hat{z}))^{(p)}$ equals \[\sum_{i=1}^p c_i \hat{z}^{i+p} S^{(i)}(\hat{z}) = D_{p} + \sum_{i=2}^p c_i\hat{z}^{i+p} S^{(i)}(\hat{z}).\]
Moreover, applying Lemma \ref{lemS}, we obtain that for all $i$ between $2$ and $p$, 
\[c_i\hat{z}^{i+p}\frac{ S^{(i)}(\hat{z})}{D_p} = \frac{c_i\hat{z}^{i-1}}{p!}\frac{S^{(i)}(\hat{z})}{S'(\hat{z})} \xrightarrow[z \rightarrow 1]{} \frac{c_i}{p!} \frac{\Gamma(-\alpha)}{\Gamma(-\alpha +1 - i )}.\]
Consequently, $(S(\hat{z}))^{(p)}$ is equivalent, when $z$ tends to $1$, to \begin{align*}
    \Bigl( 1 + \sum_{i=2}^p \frac{c_i}{p!} \frac{\Gamma(-\alpha)}{\Gamma(-\alpha +1 - i )} \Bigr) D_p
    & \underset{z \to 1}{\sim} \Bigl( p! + \sum_{i=2}^p c_i \frac{\Gamma(-\alpha)}{\Gamma(-\alpha +1 - i )} \Bigr) \hat{z}^{p+1}S'(\hat{z}) \\
    & \underset{z \to 1}{\sim} C_{\alpha,p}  \hat{z}^{p+1}S'(\hat{z}),
\end{align*}
where we set $C_{\alpha,p}=p! + \sum_{i=2}^p c_i \frac{\Gamma(-\alpha)}{\Gamma(-\alpha +1 - i )}$, so the first asymptotic equivalent follows. Note that for $p=1$ and $p=2$, $C_{\alpha,p}$ are respectively equal to $1$ and $-\alpha +1$.

Applying Lemma \ref{lemS} once more, we deduce that\[\frac{\hat{z}S'(\hat{z})}{S(\hat{z})} \xrightarrow[z \rightarrow 1]{} \frac{\Gamma(-\alpha +1)}{ \Gamma(-\alpha)} = - \alpha,\]
so $\hat{z}S'(\hat{z}) \underset{z \to 1}{\sim} -\alpha S(\hat{z})$, and the second asymptotic equivalent follows.
\end{proof}

Our goal now is to compute an equivalent of $G^a_{\mu}(z)$ as $z$ tends to $1$, in order to apply a transfer theorem to get an equivalent of $[z^{n-k-\phi_0(a)}] G^a_{\mu}(z)$ as $n$ tends to $+\infty$ and finally conclude using Lemma \ref{lemformuleRnk}.
\begin{lemma} \label{lem3.9}
For all $ a\in \mathbb{T}_k$, we have the following asymptotic equivalent
\[G^a_{\mu}(z) \underset{z \to 1}{\sim} c^{a}_{0,0,\phi_2(a)}\kappa(a)L^a(\hat{z})\hat{z}^{\psi_{0,0,\phi_2(a)}^a},\]
where $C_{\alpha,p}$ is as in Lemma \ref{lem3.7} and: 
\begin{align*} c^{a}_{i,j,p} & \coloneqq \binom{\phi_0(a)}{i}\binom{\phi_1(a)}{j}\binom{\phi_2(a)}{p}(1-\mu(0))^{\phi_0(a)-i}(1+m)^{\phi_1(a)-j}(2m)^{\phi_2(a)-p}(-\alpha)^j(\alpha^2-\alpha)^p \\
    \psi_{i,j,p}^a & \coloneqq -i\alpha - j(\alpha -1) - p(\alpha-2) - \alpha(k-\phi_0(a)-\phi_1(a) -\phi_2(a)) + (k-1-\phi_1(a) -2\phi_2(a)) \\
    \gamma_{i,j,p}^a & \coloneqq i+j+p+k-\phi_0(a)-\phi_1(a) - \phi_2(a)\\
    L^{a} & \coloneqq \ell^{\gamma_{0,0,\phi_2(a)}^a} \\
    \kappa(a) & \coloneqq\prod_{ \substack{ u \in [\![1,k]\!]\\
              c_u(a) \geq 3}} -\alpha C_{\alpha,c_u(a)}.
\end{align*}
 
\end{lemma}

\begin{proof}
We have
\begin{equation} \label{eq3.2}
    G^a_{\mu}(z) = \widetilde{F}(z)^{\phi_0(a)} F'(z)^{\phi_1(a)} F^{(2)}(z)^{\phi_2(a)}\prod_{ \substack{ u \in [\![1,k]\!]\\
              c_u(a) \geq 3}}(S(\hat{z}))^{(c_u(a))},
\end{equation}
so let's study separately each term in \eqref{eq3.2} when $z$ tends to $1$.
First, observe that $\widetilde{F}(z)= (F(z)-\mu(0))/z$, so $\widetilde{F}(z) \underset{z \to 1}{\sim} F(z)-\mu(0)$. Consequently, we get $\widetilde{F}(z)^{\phi_0(a)} \underset{z \to 1}{\sim} (1-\mu(0)+S(\hat{z}))^{\phi_0(a)}$ and recalling that $S(\hat{z})=\ell(\hat{z})\hat{z}^{-\alpha}$, we deduce the following equality 
\[\widetilde{F}(z)^{\phi_0(a)}=\sum_{i=0}^{\phi_0(a)} \binom{\phi_0(a)}{i}(1-\mu(0))^{\phi_0(a)-i}\ell(\hat{z})^{\alpha i}\hat{z}^{-\alpha i}(1+o(1)).\]
Then, $F'(z)^{\phi_1(a)}$ is equal to \[\sum_{j=0}^{\phi_1(a)} \binom{\phi_1(a)}{j}(1 - m + 2mz)^{\phi_1(a)-j} ((S(\hat{z}))^{'})^j, \]
where according to Lemma \ref{lem3.7}, $(S(\hat{z}))^{'} \underset{z \to 1}{\sim} -\alpha\hat{z}S(\hat{z}) = - \alpha \ell(\hat{z})\hat{z}^{1-\alpha}$. Thus, we obtain that\[F'(z)^{\phi_1(a)} = \sum_{j=0}^{\phi_1(a)} \binom{\phi_1(a)}{j}(1 +m )^{\phi_1(a)-j}(- \alpha)^j \ell(\hat{z})^j \hat{z}^{-j(\alpha-1)}(1 +o(1)).\]
Finally, applying again Lemma \ref{lem3.7}, we get 
\begin{align*}
        F^{(2)}(z)^{\phi_2(a)} & = \sum_{p=0}^{\phi_2(a)} \binom{\phi_2(a)}{p} (2m)^{\phi_2(a)-p} ((S(\hat{z}))^{(2)})^p \\
        & = \sum_{p=0}^{\phi_2(a)} \binom{\phi_2(a)}{p} (2m)^{\phi_2(a)-p} (-\alpha C_{\alpha,2 })^p \ell(\hat{z})^p \hat{z}^{-p(\alpha-2)}(1 +o(1))
    \end{align*}
and
\begin{align*}
    \prod_{ \substack{ u \in [\![1,k]\!]\\
              c_u(a) \geq 3}}(S(\hat{z}))^{(c_u(a))} & = \prod_{ \substack{ u \in [\![1,k]\!]\\
              c_u(a) \geq 3}} -\alpha C_{\alpha,c_u(a)} \ell(\hat{z})\hat{z}^{-\alpha+c_u(a)}(1 +o(1)).
\end{align*}
Hence, by injecting these relations into Equation \eqref{eq3.2}, we get
\[G^a_{\mu}(z) = \kappa(a) \sum_{i=0}^{\phi_0(a)}\sum_{j=0}^{\phi_1(a)} \sum_{p=0}^{\phi_2(a)} c^{a}_{i,j,p}\ell(\hat{z})^{\gamma_{i,j,p}^a} \hat{z}^{\psi_{i,j,p}^a} (1 + o(1)).\]

Now, let's prove that there is a dominant term in this sum. First, since $1<\alpha<2$, the value of $\psi_{i,j,p}^a$ decreases with $i$ and $j$ and increases with $p$ so it uniquely achieves its maximum at $(0,0,\phi_2(a))$. Therefore, we can fix $\epsilon > 0$ such that
\begin{equation} \label{ineq1}
    \forall (i,j,p) \in [\![0,\phi_0(a)]\!]\times[\![0,\phi_1(a)]\!]\times[\![0,\phi_2(a)]\!]\setminus{(0,0,\phi_2(a))} \textbf{ , } \psi_{0,0,\phi_2(a)}^a -\epsilon > \psi_{i,j,p}^a + \epsilon.
\end{equation}
According to Theorem 1.5.6 (Potter's bounds) in \cite{BGT97}, for all $(i,j,p) \in [\![0,\phi_0(a)]\!]\times[\![0,\phi_1(a)]\!]\times[\![0,\phi_2(a)]\!]$, there exist $x_0$, $C^-_{i,j,p}$, $C^+_{i,j,p}$ such that for all $ \hat{z} \geq x_0$, 
\begin{equation*} 
    C^-_{i,j,p}\hat{z}^{-\epsilon} < \ell(\hat{z})^{\gamma^a_{i,j,p}} < C^+_{i,j,p}\hat{z}^{\epsilon}.
\end{equation*}
Then, for all $(i,j,p) \in [\![0,\phi_0(a)]\!]\times [\![0,\phi_1(a)]\!]\times[\![0,\phi_2(a)]\!]\setminus{(0,0,\phi_2(a))}$, 
\[\frac{\ell(\hat{z})^{\gamma^a_{i,j,p}}\hat{z}^{\psi_{i,j,p}^a}}{\ell(\hat{z})^{\gamma^a_{0,0,\phi_2(a)}}\hat{z}^{\psi_{0,0,\phi_2(a)}^a}}< \frac{C^+_{i,j,p}\hat{z}^{\psi_{i,j,p}^a+\epsilon}}{C^-_{0,0,\phi_2(a)}\hat{z}^{\psi_{0,0,\phi_2(a)}^a-\epsilon}} \xrightarrow[z \rightarrow 1]{} 0, \]
by Equation \eqref{ineq1}. So for all $(i,j,p) \in [\![0,\phi_0(a)]\!]\times[\![0,\phi_1(a)]\!]\times[\![0,\phi_2(a)]\!]\setminus{(0,0,\phi_2(a))}$, 
\[\ell(\hat{z})^{\gamma^a_{i,j,p}}\hat{z}^{\psi_{i,j,p}^a} = \underset{z \to 1}{o}\bigl(\ell(\hat{z})^{\gamma_{0,0,\phi_2(a)}^a}\hat{z}^{\psi_{0,0,\phi_2(a)}^a}\bigr),\]
and the lemma follows.
\end{proof}
 
\begin{lemma}\label{lem4.9}  We have the following asymptotic equivalent
\[ \sum_{a \in \mathbb{T}_k} \biggl(\zeta_k(a)\left[ z^{n-k-\phi_0(a)} \right] G^a_{\mu}(z) \biggr) \underset{n \to \infty}{\sim} \frac{-\alpha C_{\alpha,k-1}}{\Gamma(k)\Gamma(-\alpha+k-1)} \ell(n) n^{ k - \alpha  -2} \]
\end{lemma}

\begin{proof}
We have that $\psi_{0,0,\phi_2(\ast_k)}^{\ast_k} = k - \alpha -1 \notin \mathbb{Z}_{\geq0}$ for $1< \alpha < 2$, so by Lemma \ref{lem3.9}, we can apply Theorem 5 of \cite{FO90} to get that
\begin{equation}\label{eqtransfert1}
   \zeta_k(\ast_k)\left[ z^{n-k-\phi_0(\ast_k)} \right] G^{\ast_k}_{\mu}(z) \underset{n \to \infty}{\sim} \frac{\zeta_k(\ast_k) \kappa(\ast_k)}{\Gamma(-\psi_{0,0,\phi_2(\ast_k)}^{\ast_k})} \frac{L^{\ast_k}(n)}{n^{\psi_{0,0,\phi_2(\ast_k)}^{\ast_k} + 1}},
 \end{equation}
with $\zeta_k(\ast_k)= \Gamma(k)^{-1} $, $\phi_0(\ast_k) = k-1 $, $\psi_{0,0,\phi_2(\ast_k)}^{\ast_k} = \alpha -k +1$, $\kappa(\ast_k) = -\alpha C_{\alpha,k-1}$ and $L^{\ast_k}(n)=\ell(n)$. Moreover, applying again Theorem 5 of \cite{FO90} for all $ a \in \mathbb{T}_k, a \neq \ast_k$, we obtain
 \begin{equation}\label{eqtransfert2}
    \zeta_k(a) \left[ z^{n-k-\phi_0(\ast_k)} \right] G^{\ast_k}_{\mu}(z) =  \underset{n \to \infty}{o} \biggl(L^a(n)n^{\psi_{0,0,\phi_2(a)}^a-1}\biggr).
 \end{equation}
So, by adding Equations \eqref{eqtransfert1} and \eqref{eqtransfert2}, we get that $\sum_{a \in \mathbb{T}_k} \zeta_k(a)\left[ z^{n-k-\phi_0(a)} \right] G^{a}_{\mu}(z)$ is equal to
\begin{align*}
      \frac{-\alpha C_{\alpha,k-1}}{\Gamma(k)\Gamma(-\alpha+k-1)} \ell(n)n^{k- \alpha- 2} + o\Bigl(\ell(n)n^{k- \alpha- 2} \Bigr) + \sum_{\substack{ \text{ }a \in \mathbb{T}_k\\
               a \neq \ast_k}}o\Bigl(L^a(n)n^{\psi_{0,0,\phi_2(a)}^a-1} \Bigr).
\end{align*}
Since $\ell$ and $L^a$ are slowly varying, they can be bounded above and below by powers of $n$ that are arbitrarily close to $0$, see e.g. Theorem 1.5.6 in \cite{BGT97}. As we have $\psi_{0,0,\phi_2(\ast_k)}^{\ast_k} = k - \alpha - 1 $, our claim follow if we prove that $\psi_{0,0,\phi_2(a)}^a < \psi_{0,0,\phi_2(\ast_k)}^{\ast_k}$ for all $a \in \mathbb{T}_k \setminus \{\ast_k\}$. 

For all $a\in \mathbb{T}_k$, we define $f(a) =  \alpha\phi_0(a) + (\alpha - 1)\phi_1(a)$, so $\psi_{0,0,\phi_2(a)}^a = -k(\alpha -1) - 1 + f(a)$.
In order to maximize the function $f$, we define a transformation $tr : \mathbb{T}_k \rightarrow \mathbb{T}_k$ satisfying 
\begin{equation*}
    tr(\ast_k)=\ast_k, \quad \forall a \in \mathbb{T}_k \setminus \{\ast_k\} \text{, } f(tr(a)) < f(a) \quad \text{and }\quad \exists s' \in \mathbb{Z}_{\geq0} : \forall s \geq s', tr^{s}(a) = \ast_k.
\end{equation*}
Let us describe $tr$: If $a=\ast_k$ then $tr(a)=a$. If $a \neq \ast_k$, then there exists a vertex $u \neq \varnothing_a $ such that $c_u(a) \geq 1.$ We note $v$ the first one in the lexicographic order satisfying that. We get $tr(a)$ from $a$ by doing the following changes: We attach the offspring of $v$ on the right of the offspring of $\varnothing_a $ and we put $c_v(a) = 0$. We give an example of iterations of the transformation $tr$ in Figure \ref{Fig6}.
\begin{figure}[h] 
    \centering
    \includegraphics[scale = 1]{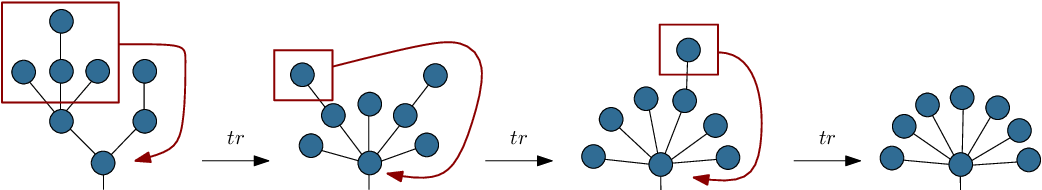}
    \caption{An example of iterations of $tr$.}
    \label{Fig6}
\end{figure}

Let $a\in \mathbb{T}_k$. We have that $\phi_0(tr(a)) = \phi_0(a) + 1$ and \[\phi_1(tr(a)) = \left\{
    \begin{array}{ll}
        \phi_1(a) - 2 & \mbox{if } c_v(a)=c_{\varnothing_a}(a)=1 \\
        \phi_1(a) - 1 & \mbox{if } (c_v(a)=1,c_{\varnothing_a}(a)\neq 1 ) \mbox{ or } (c_v(a)\neq1,c_{\varnothing_a}(a)=1 )\\
        \phi_1(a) & \mbox{else.} 
    \end{array}
\right.\]
Consequently, one can check that in all cases, $f(tr(a)) > f(a)$ for all $a \in \mathbb{T}_k \setminus \{\ast_k\}$, so the lemma follows.
\end{proof}

\begin{proof}[Proof of Theorem \ref{th3.2}.1]
Combining Equation (\ref{eqtransfert1}),(\ref{eqtransfert2}), Lemma \ref{lemformuleRnk} and \ref{lem4.9}, we get
\[\mathbb{P}(R_{n,k} = a) \xrightarrow[n \rightarrow +\infty]{} \left\{
    \begin{array}{ll}
        1  & \mbox{if $a = \ast_k$}\\
        0  & \mbox{otherwise,}
    \end{array}
\right. \]
which concludes the proof of the first point of Theorem \ref{th3.2}. 
\end{proof}

\subsection{The root has maximal degree}
We now aim to prove Theorem \ref{th3.2}.2. Recall that we denote by $c^o_i(a)$ the number of children of the $i$-th internal vertex, counted in lexicographic order, in a tree $a$.

\begin{proposition}\label{prop3.11} There exists $\delta > 0 $  such that
    \[\bigg|\frac{c^o_0(T_{n,k})}{n}-1\bigg|n^{\delta}\xrightarrow[n \rightarrow +\infty]{\mathbb{P}} 0. \]
\end{proposition}

Observe that this implies $n^{\delta-1} \max_{1 \leq i \leq k-1}c^o_i(T_{n,k})\xrightarrow[n \rightarrow +\infty]{\mathbb{P}} 0$. Moreover, Proposition \ref{prop3.11} is in fact stronger than Theorem \ref{th3.2}.2. We will see in the proof why such a stronger statement is required.

In order to prove Proposition \ref{prop3.11}, we start by proving that the largest degree concentrates around $n$.

\begin{lemma}\label{lem3.13} There exists $ \delta > 0 $  such that
     \[ \biggl(\frac{1}{n}\max_{0 \leq i \leq k-1}c^o_i(T_{n,k}) - 1\biggr)n^{\delta} \xrightarrow[n \rightarrow +\infty]{\mathbb{P}} 0.\]
\end{lemma}
\begin{proof} Since $\alpha > 1$ we may fix $ 0< \gamma <1 $ such that $ 1 + 2\gamma \leq  \alpha$. Set $\delta \coloneqq (\alpha -1 -2\gamma) /2\alpha >0$. Since $R_{n,k}=\ast_k$ with probability tending to $1$, by Theorem \ref{th3.2}.1, it is enough to prove that $(\max_{0 \leq i \leq k-1}c^o_i(T_{n,k})/n -1)n^{\delta}$ converges in probability to $0$ under the conditional probability $\mathbb{P}(\cdot|R_{n,k}=\ast_k)$ , i.e. that for all $\epsilon >0$, \[\mathbb{P}\biggl( \max_{0 \leq i \leq k-1}c^o_i(T_{n,k}) \leq (1-\epsilon n^{-\delta})n \,\Big| \,R_{n,k}=\ast_k \biggr) \xrightarrow[n \rightarrow +\infty]{} 0. \]
Set  $\epsilon_n \coloneqq \epsilon/n^{\delta}$, so that $\epsilon^{2\alpha} n^{1-\alpha+\gamma} = \epsilon_n^{2\alpha} n^{-\gamma }$. Then this probability equals
\begin{align*}
    & \mathbb{P}\Bigl( c^o_0(T_{n,k}) \leq (1-\epsilon_n)n ,\ldots,c^o_{k-1}(T_{n,k}) \leq (1-\epsilon_n)n \,\big| \,R_{n,k}=\ast_k\Bigr) \\
    & = \frac{1}{\mathbb{P}(R=\ast_k, T \in \mathbb{T}_{n,k})}\mathbb{P}\Bigl(c^o_0(T) \leq (1-\epsilon_n)n ,\ldots,c^o_{k-1}(T) \leq (1-\epsilon_n)n , R=\ast_k, T \in \mathbb{T}_{n,k} \Bigr) \\
    & =\frac{1}{\mathbb{P}(R=\ast_k, T \in \mathbb{T}_{n,k})} \sum_{\substack{ 1 \leq i_0,\ldots,i_{k-1} \leq n(1-\epsilon_n)\\ \sum i_j = n-1}} \mathbb{P}\Bigl( c^o_{0}(T)= i_0,\ldots,c^o_{k-1}(T)= i_{k-1}, R=\ast_k, T \in \mathbb{T}_{n,k} \Bigr) \\ 
    & = \frac{1}{\mathbb{P}(R=\ast_k, T \in \mathbb{T}_{n,k})} \sum_{\substack{ 1 \leq i_0,\ldots,i_{k-1} \leq n(1-\epsilon_n)\\ \sum i_j = n-1}} \binom{i_0}{k-1} \mu(i_0)\mu(i_1)\cdots\mu(i_{k-1})\mu(0)^{n-k}.
\end{align*}
According to Lemma \ref{lemformuleRnk} and Equation \eqref{eqtransfert1}, by setting $c_k \coloneqq \frac{\kappa(\ast_k)}{(k-1)!\Gamma(k-\alpha-1)}$, we have \[\mathbb{P}(R_{n,k}=\ast_k, T \in \mathbb{T}_{n,k}) \underset{n \to \infty}{\sim} c_k\mu(0)^{n-k}\ell(n)n^{ k -\alpha - 2}.\] Fix $\eta > 0$ arbitrary. Then for $n$ large enough, we get
\[\frac{\mu(0)^{n-k}}{\mathbb{P}(R_{n,k}=\ast_k, T \in \mathbb{T}_{n,k})} \leq \frac{1}{(1-\eta)c_k \ell(n) n^{ k -\alpha - 2}}.\]
Furthermore, we have $ \binom{i_0}{k-1} \leq n^{k-1} $ for every $ i_0 \leq n $ and 
\[\sum_{\substack{1 \leq i_0,...,i_{k-1} \leq n(1-\epsilon_n)\\ \sum i_j = n-1}} \mu(i_0)\mu(i_1)\cdots\mu(i_{k-1}) = \mathbb{P}\Bigl(X_0+\cdots+X_{k-1}=n-1, \forall j \in [\![0,k-1]\!] \textbf{ } 1 \leq X_j \leq n(1-\epsilon_n) \Bigr) \]
where the $X_j$'s are i.i.d. with law $\mu$. Observe that on the last event, at least two different $X_j$'s must be greater than $n\epsilon_n/2k$. Thus, we have for $n$ large enough,
\begin{align*}
   \mathbb{P}\biggl( \biggl| \frac{1}{n} \max_{0 \leq i \leq k-1}c^o_i(T_{n,k}) -1 \biggl| \geq \epsilon_n \, \Bigl| \,R_{n,k}=\ast_k\biggr)
   & \leq \frac{1}{(1-\eta)c_k} \frac{n^{\alpha + 1} }{\ell(n)}  \mathbb{P}\biggl(\exists i\neq j  \text{ st }  X_i, X_j \geq \frac{n\epsilon_n }{2k} \biggr) \\
   & \leq \frac{k^2}{(1-\eta)c_k} \frac{n^{\alpha + 1} }{\ell(n)}\mathbb{P}\biggl( X \geq \frac{n\epsilon_n }{2k} \biggr)^2.
\end{align*}
Recall that $\mathbb{P}( X \geq x) \underset{x \to \infty}{\sim} -\Gamma(1-\alpha)^{-1}\ell(x)x^{-\alpha}$ by Equation \eqref{propmu}, then
 \[\frac{n^{\alpha + 1} }{\ell(n)}\mathbb{P}\biggl( X \geq \frac{n\epsilon_n }{2k} \biggr)^2 \underset{n \to \infty}{\sim} \frac{1}{(\Gamma(1-\alpha))^2}\frac{ \ell(n\epsilon_n/2k)^2}{\ell(n)}  \frac{ n^{\alpha + 1}}{(n\epsilon_n/2k)^{2\alpha}}.\]
Moreover, since $\ell$ is slowly varying, we can apply Potter's bounds: there exists $C>0$ and $N>0$ such that for $n > N$, $\ell(n\epsilon_n/2k)^2/\ell(n)\leq C n^{\gamma}$. Therefore, we deduce that for every $n > N$,
\[ \frac{n^{\alpha + 1} }{\ell(n)}\mathbb{P}\biggl( X \geq \frac{n\epsilon_n }{2k} \biggr)^2 \leq \frac{C}{(\Gamma(1-\alpha))^2}\frac{ n^{\alpha + 1 + \gamma}}{(n\epsilon_n/2k)^{2\alpha}} = \frac{C}{(\Gamma(1-\alpha))^2(2k\epsilon)^{2\alpha}} n^{-\gamma} \xrightarrow[n \rightarrow +\infty]{} 0.\]
Finally, \[\mathbb{P}\biggl( \biggl| \frac{1}{n} \max_{0 \leq i \leq k-1}c^o_i(T_{n,k}) -1 \biggl| \geq \epsilon_n \,\Bigl| \,R_{n,k}=\ast_k\biggr) \xrightarrow[n \rightarrow +\infty]{} 0, \]
and the claim follows.
\end{proof}
We can now proceed to the proof of Proposition \ref{prop3.11}.
\begin{proof}[Proof of Proposition \ref{prop3.11}]
Set $\epsilon > 0$ and pick the same $\gamma$, $\delta$ and $\epsilon_n=\epsilon n^{-\delta}$ as in the proof of Lemma \ref{lem3.13}. For every $i \in [\![0,k-1]\!] $, we have
\[\mathbb{P}\Bigl( T \in \mathbb{T}_{n,k}, R=\ast_k, c^o_i(T) \geq (1-\epsilon_n)n \Bigr) = \sum_{\substack{a_0+\cdots+a_{k-1}=n-1\\ a_i \geq (1-\epsilon_n)n}} \binom{a_0}{k-1} \mu(a_0)\cdots\mu(a_{k-1})\mu(0)^{n-k}.\]
For $i=0$, we impose $a_0 \geq (1-\epsilon_n)n$ so $\binom{a_0}{k-1} \geq \binom{(1-\epsilon_n)n}{k-1}$ whereas for $i\geq 1$ we have $a_0 \leq \epsilon_nn$ so $\binom{a_0}{k-1} \leq \binom{\epsilon_nn}{k-1}$. Therefore
\[\frac{\mathbb{P}(T \in \mathbb{T}_{n,k}, R=\ast_k, \exists i \geq 1 \text{: } c^o_i(T) \geq (1-\epsilon_n)n) }{\mathbb{P}(T \in \mathbb{T}_{n,k}, R=\ast_k, c^o_0(T) \geq (1-\epsilon_n)n)}  \leq (k-1)\frac{\binom{\epsilon_nn}{k-1}}{\binom{(1-\epsilon_n)n}{k-1}} \xrightarrow[n \rightarrow +\infty]{} 0.\]
Using that $b/(a+b) \rightarrow 1$ when $a/b \rightarrow 0$, we obtain
\[\frac{\mathbb{P}(T \in \mathbb{T}_{n,k}, R=\ast_k, c^o_0(T) \geq (1-\epsilon_n)n ) }{\mathbb{P}(T \in \mathbb{T}_{n,k}, R=\ast_k,\exists i\geq 0 \text{: } c^o_i(T) \geq (1-\epsilon_n)n)}\xrightarrow[n \rightarrow +\infty]{} 1.\]
Dividing both the numerator and denominator by $\mathbb{P}\bigl(T \in \mathbb{T}_{n,k}, R=\ast_k \bigr)$, we observe that the conditional probability in the denominator goes to $1$ by Lemma \ref{lem3.13}. Therefore, the conditional probability in the numerator also converges to $1$. The claim then follows from Theorem \ref{th3.2}.1. \end{proof}

\section{Transfer case: Proof of Theorem \ref{th_cvg_transfert_case-intro}} \label{sec7}
This section is devoted to the proof of Theorem \ref{th_cvg_transfert_case-intro}. We consider an offspring distribution $\mu$ whose generating function $F_{\mu}$ has radius of convergence $\rho > 1$ and satisfies the following assumption for some $\alpha > 0$:
\begin{equation}\label{halpha}
``F_{\mu}(z)\text{ is } \Delta\text{-analytic } \text{and satisfies } F_{\mu}(z)\underset{z \to \rho}{\sim} \frac{c}{(1-z/\rho)^{\alpha}} \text{, }c \in \mathbb{R}",
    \tag{$\mathcal{H}_{\alpha}$}
\end{equation}
where we recall from e.g. \cite{F13} (Definition VI.1) that a function $G$ is $\Delta$-analytic when there exist a radius $R>\rho$ and an angle $\phi \in (0,\pi/2)$ such that $G$ is analytic on the domain $\{z:|z|<R,z\neq\rho,|\arg(z-\rho)|>\phi\}$.
Theorem \ref{th_cvg_transfert_case-intro} describes the asymptotic behavior of a $\mu$-BGW tree conditioned to have $n$ vertices and $k$ internal nodes, under the assumption that $F_{\mu}$ satisfies \eqref{halpha}, by describing the behavior of its reduced tree and the associated sequence of leaves (as explained in Remark \ref{rqbij2}).
It is interesting to note that by Lemma 4.1 of \cite{JJ12}, the limit of the reduced tree has the distribution of a $\nu$-BGW tree conditioned to have $k$ vertices, where the generating function of $\nu$ is given by $F_{\nu}(z)=(\alpha/(1+\alpha -z))^{\alpha}$.

\begin{proof}[Proof of Theorem \ref{th_cvg_transfert_case-intro}] We begin by proving the first assertion of Theorem \ref{th_cvg_transfert_case-intro}. Set $\alpha >0$ and $a \in \mathbb{T}_k$. Recall that $\widetilde{F}_{\mu}$ is the generating function of $\mu(\cdot+1)$ and that $G^a_{\mu}$ is defined as follows \[G^a_{\mu} (z)\coloneqq \prod_{u \notin I(a)}\widetilde{F}_{\mu} (z)\prod_{u \in I(a)} F_{\mu}^{(c_u(a))}(z). \]
The generating function $F_{\mu}$ satisfies \eqref{halpha}, so does $\widetilde{F}(z)$ as $\widetilde{F}(z)= (F(z)-\mu(0))/z \underset{z \to \rho}{\sim}(1-z/\rho)^{-\alpha}c/\rho$. Then, applying Theorem VI.8 of \cite{F13} we get that $G^a_{\mu}$ is $\Delta$-analytic as a product and derivatives of $\Delta$-analytic functions, and replacing dominations by negligibilities in this theorem, we obtain:
\[G^a_{\mu}(z) \underset{z \to \rho}{\sim} \frac{1}{\rho^{\phi_0(a)+k-1}} \frac{c^k}{\Gamma(\alpha)^k}\biggl( \prod_{u=1}^k \Gamma(\alpha + c_u(a)) \biggr) \bigl(1-z/\rho\bigl)^{-\alpha k-k+1}. \]
Thus by Corollary VI.1 of \cite{F13} we deduce that
\begin{equation}\label{eq1transfertcase}
    [z^{n-k-\phi_0(a)}]G^a_{\mu} \underset{n \to \infty}{\sim} \frac{c^k}{\rho^{n-1}\Gamma(\alpha)^k} \biggl( \prod_{u=1}^k \Gamma(\alpha + c_u(a)) \biggr) \frac{\bigl(n-k-\phi_0(a)\bigl)^{\alpha k+k-2}}{\Gamma(\alpha k +k-1)}.
\end{equation}
Recall from Lemma \ref{lemformuleRnk} that $\mathbb{P}(R_{n,k}=a)$ is proportional to $\zeta_k(a) [z^{n-k-\phi_0(a)}]G^a_{\mu}$. Then, we conclude that 
\[\mathbb{P}\bigl(R_{n,k}=a\bigr)\xrightarrow[n \rightarrow \infty]{}\frac{\prod_{u=1}^k w_u(a) }{\sum_{b \in \mathbb{T}_k} \prod_{u=1}^k w_u(b)},\]
and the first point of Theorem \ref{th_cvg_transfert_case-intro} follows.

We now turn to the proof of the second part of Theorem \ref{th_cvg_transfert_case-intro}. The argument is divided into two lemmas. We begin by establishing the first key estimate, stated in the following lemma.

\begin{lemma}\label{lemma1transfert}
    Set $\alpha >0$ and $a \in \mathbb{T}_k$. Let $\mu$ be an offspring distribution with generating function satisfying \eqref{halpha}. Then,
\begin{enumerate}
    \item conditionally given $ R_{n,k}=a$, we have the following convergence in distribution
    \[\frac{\bigl(X^n_1, \ldots, X^n_{k}\bigr)}{n} \xrightarrow[n \rightarrow +\infty]{(d)} \mathbf{Dir}\bigl(\alpha + c_1(R_{n,k}),\ldots,\alpha + c_k(R_{n,k})\bigr),\]
    where $X^n_i$ is the number of leaves attached to the $i$-th internal node of $T_{n,k}$. 
    \item conditionally given $ R_{n,k}=a$ and $(X^n_1, \ldots, X^n_{k})$, for every $i\in [\![1,k]\!] $, the vector \[\Bigl(L_{1+(\sum_{j=1}^{i-1}c_{j}(a)+1)}\bigl(T_{n,k}\bigr),\ldots, L_{c_i(a)+1+(\sum_{j=1}^{i-1}c_{j}(a)+1)}\bigl(T_{n,k}\bigr)\Bigr)\] is a composition of $n-k-(\sum_{j\neq i} X^n_j)$ into $c_i(a)+1$ parts sampled uniformly at random.
\end{enumerate}
\end{lemma}

Finally, the second lemma explains how to combine the two points from Lemma \ref{lemma1transfert} to conclude the proof of Theorem \ref{th_cvg_transfert_case-intro}.

\begin{lemma}\label{lemma2transfert}
    Let $\alpha_1,\ldots,\alpha_n \in \mathbb{R^*_+}$. Set $X=(X_i)_{1\leq i\leq n}$ and $Y=(Y_i)_{1\leq i \leq k}$ random variables such that $X$ and $Y$ are independent,  $(X_i)_{1\leq i\leq n} \sim \mathbf{Dir}(\alpha_1,\ldots,\alpha_n)$ and $Y \sim \mathbf{Dir}(1,\ldots,1)$. Then, we have that:
    \[\bigl( X_1Y,X_2,\ldots,X_n\bigr) \sim \mathbf{Dir}\biggl(\Bigl(\frac{\alpha_1}{k}\Bigr)_{1\leq i \leq k}, \alpha_2,\ldots,\alpha_n\biggr). \]
\end{lemma}
Combining these two lemmas, the second point of Theorem \ref{th_cvg_transfert_case-intro} follows directly. \end{proof}
It remains to prove Lemma \ref{lemma1transfert} and \ref{lemma2transfert}.
\begin{proof}[Proof of Lemma \ref{lemma1transfert}]
Fix $a \in \mathbb{T}_k$. Set $x_1,\ldots,x_k \in \mathbb{R}$ and $y_1,\ldots,y_k \in \mathbb{Z}_{\geq0}$ such that $\sum_iy_i=n-k$ and such that $y_i/n$ tends to $x_i$ when $n$ goes to infinity. We define $k$ random variables $X^n_1,\ldots,X^n_k$ such that $X^n_i$ is the number of leaves of the $i$-th internal node in $T_{n,k}$. We have that \[\mathbb{P}\bigl((X^n_1,\ldots,X^n_k)=(y_1,\ldots,y_k)\,\big| R_{n,k}=a\,\bigr) = \frac{\mathbb{P}\bigl((X^n_1,\ldots,X^n_k)=(y_1,\ldots,y_k),R=a,T \in \mathbb{T}_{n,k}\bigr)}{\mathbb{P}\bigl(R=a,T \in \mathbb{T}_{n,k}\bigl)},\]
where, combining Lemma \ref{lemformuleRnk} and Equation \eqref{eq1transfertcase}, the denominator satisfies
\[\mathbb{P}(R=a,T \in \mathbb{T}_{n,k})\underset{n \to \infty }{\sim} \frac{\zeta_k(a)c^k\mu(0)^{n-k}}{\rho^{n-1}\Gamma(\alpha k +k-1)\Gamma(\alpha)^k} \biggl( \prod_{u=1}^k \Gamma(\alpha + c_u(a)) \biggr) n^{\alpha k+k-2}.\] Now let us compute the numerator. We have that
 \[\mathbb{P}\bigl((X^n_1,\ldots,X^n_k)=(y_1,\ldots,y_k),R=a,T \in \mathbb{T}_{n,k}\bigr) = \mu(0)^{n-k}\prod_{u=1}^k \binom{y_u+c_u(a)}{c_u(a)}\mu(y_u+c_u(a)),\]
 and by Corollary VI.1 of \cite{F13}, we get that this is equivalent as $n$ tends to infinity to 
 \[\mu(0)^{n-k}\prod_{u=1}^k \binom{y_u+c_u(a)}{c_u(a)}\frac{c(y_u+c_u(a))^{\alpha-1}}{\Gamma(\alpha)\rho^{y_u+c_u(a)}} \underset{n \to \infty }{\sim} \frac{c^k\mu(0)^{n-k}}{\Gamma(\alpha)^k \rho^{n-1}}\zeta_k(a) \prod_{u=1}^k y_u^{c_u(a)+\alpha-1}. \]
 As $y_i\sim nx_i$, we deduce that
 \[\mathbb{P}\bigl((X^n_1,\ldots,X^n_k)=(y_1,\ldots,y_k),R=a,T \in \mathbb{T}_{n,k}\bigr)\underset{n \to \infty }{\sim} \frac{c^k\mu(0)^{n-k}}{\Gamma(\alpha)^k \rho^{n-1}}\zeta_k(a) n^{\alpha k-1}\prod_{u=1}^k x_u^{c_u(a)+\alpha-1}.\]
 Finally, we obtain that 
 \[\mathbb{P}\bigl((X^n_1,\ldots,X^n_k)=(y_1,\ldots,y_k)\,\big| R_{n,k}=a\,\bigr) \underset{n \to \infty }{\sim} \frac{1}{n^{k-1}}\frac{\Gamma \bigl( \sum_{u=1}^k c_u(a) + \alpha\bigr)}{\prod_{u=1}^k\Gamma(c_u(a) + \alpha)} x_1^{c_1(a)+\alpha-1} \cdots x_k^{c_k(a)+\alpha-1}.\]
 This local estimate implies the desired convergence in distribution so we get the lemma. 
\end{proof}

\begin{proof}[Proof of Lemma \ref{lemma2transfert}]
The Dirichlet distribution is characterized by its moments, so it is enough to prove that for all integers $\lambda_1^1,\ldots,\lambda_1^k,\lambda_2,\ldots,\lambda_n$, we have the following equality \[\mathbb{E}\Bigl( \bigl(X_1Y_1\bigl)^{\lambda_1^1}\cdots\bigr(X_1Y_k\bigl)^{\lambda_1^k}X_2^{\lambda_2}\cdots X_n^{\lambda_n} \Bigr) = \mathbb{E}\Bigl( U_1^{\lambda_1^1}\cdots U_k^{\lambda_1^k}X_2^{'\lambda_2}\cdots X_n^{'\lambda_n} \Bigr),\]
where $(U_1,\ldots,U_k,X'_2,\ldots,X'_n)\sim \mathbf{Dir}((\alpha_1/k)_{1\leq i \leq k}, \alpha_2,\ldots,\alpha_n)$.
As $Y$ is independent of $X$, we have that \[\mathbb{E}\Bigl( \bigl(X_1Y_1\bigl)^{\lambda_1^1}\cdots \bigr(X_1Y_k\bigl)^{\lambda_1^k} X_2^{\lambda_2} \cdots X_n^{\lambda_n} \Bigr) = \mathbb{E}\Bigl(Y_1^{\lambda_1^1}\cdots Y_k^{\lambda_1^k}\Bigl)\mathbb{E}\Bigl(X_1^{\lambda_1^1+\cdots+\lambda_1^k}X_2^{\lambda_2}\cdots X_n^{\lambda_n} \Bigr),\]
with
\begin{align*}
\mathbb{E}\Bigl(Y_1^{\lambda_1^1}\cdots Y_k^{\lambda_1^k}\Bigl) & = \frac{\Gamma(k)\prod_{i=1}^k\Gamma(\lambda_1^i+1)}{\Gamma\bigl(\sum_{i=1}^k(\lambda_1^i+1)\bigr)}, \\
\mathbb{E}\Bigl(X_1^{\lambda_1^1+\cdots+\lambda_1^k} X_2^{\lambda_2}\cdots X_n^{\lambda_n} \Bigr) & = \frac{\Gamma\bigl(\sum_{j=1}^n\alpha_j\bigr) \Gamma\bigl(\alpha_1+\sum_{i=1}^k\lambda_1^i\bigr)\prod_{j=2}^n\Gamma(\alpha_j+\lambda_j)}{\Gamma\bigl(\alpha_1+\sum_{i=1}^k\lambda_1^i+\sum_{j=2}^n(\lambda_j+\alpha_j)\bigr)\prod_{j=1}^n\Gamma(\alpha_j)}.
\end{align*}

As \[\mathbb{E}\Bigl( U_1^{\lambda_1^1} \cdots U_k^{\lambda_1^k} X_2^{'\lambda_2} \cdots X_n^{'\lambda_n} \Bigr)=\frac{\Gamma\bigl(\sum_{j=1}^n\alpha_j\bigr)\prod_{i=1}^k\Gamma\bigl(\lambda_1^i+\frac{\alpha_1}{k}\bigr)\prod_{j=2}^n\Gamma\bigl(\lambda_j+\alpha_j\bigr)}{\Gamma\bigl(\alpha_1+\sum_{i=1}^k\lambda_1^i+\sum_{j=2}^n(\lambda_j+\alpha_j)\bigr)\Gamma\bigl(\frac{\alpha_1}{k}\bigr)\prod_{j=2}^n\Gamma\bigl(\alpha_j\bigr)},\]
using that for all real $z$, $\Gamma(z+1)=z\Gamma(z)$, we get the desired equality. 
\end{proof}

\section{Poisson case: Proof of Theorem \ref{thPois-intro}} \label{sec8}
The aim of this section is to prove Theorem \ref{thPois-intro}. We consider $\mu$ an offspring distribution with generating function satisfying the assumption \eqref{hp} defined by
\begin{equation}\label{hp}
\text{``}F_{\mu}(z) = c\exp{(P(z))} \text{ where } c>0, P(z) \coloneqq \sum_{i=1}^pa_iz^i, a_i \geq 0 \text{ and } \mathrm{gcd}\{j:a_j\neq0\}=1 \text{''}
\tag{$\mathcal{H}_{P}$}
\end{equation}

Theorem \ref{thPois-intro} shows that under assumption \eqref{hp}, the reduced tree of a $\mu$-BGW tree conditioned to have $n$ vertices and $k$ internal nodes converges in distribution to an explicit simply generated tree, along with the way the leaves are attached to it.

The third assertion of Theorem \ref{thPois-intro} is an immediate consequence of the first two and the definition of a BGW tree.

To establish the first point, it suffices to prove the following result.
\begin{proposition}\label{proppois}
Let $\mu$ be an offspring distribution with generating function satisfying \eqref{hp}. Then for all $a,b \in \mathbb{T}_k $, \[[z^{n-k-\phi_0(a)}]G^a_{\mu}(z) \underset{n \to \infty}{\sim} [z^{n-k-\phi_0(b)}]G^b_{\mu}(z). \]
\end{proposition}
\begin{proof}[Proof of Theorem \ref{thPois-intro}.1]
Let us assume this result and see how it leads to the first point of Theorem \ref{thPois-intro}. Applying Lemma \ref{lemformuleRnk}, we get that for all $a \in \mathbb{T}_k$
\[\mathbb{P}(R_{n,k} = a) \underset{n \to \infty}{\sim} \frac{\zeta_k(a)[z^{n-k-\phi_0(a)}]G^a_{\mu}(z)}{\sum_{b \in \mathbb{T}_k} \zeta_k(b)[z^{n-k-\phi_0(b)}]G^b_{\mu}(z)} \xrightarrow[n \rightarrow \infty]{} \frac{\zeta_k(a)}{\sum_{b \in \mathbb{T}_k} \zeta_k(b)},\]
so we get the result.
\end{proof}
Now, let us prove Proposition \ref{proppois}. We divide the proof of Proposition \ref{proppois} into two lemmas.
\begin{lemma} \label{lempolyetF}
Let $\mu$ be an offspring distribution with generating function satisfying \eqref{hp}.
\begin{enumerate}
    \item For all $a \in \mathbb{T}_k$, \[G^a_{\mu}(z) = Q_a(z) \sum_{j=0}^{\phi_0(a)} \binom{\phi_0(a)}{j}\mu(0)^jF_{\mu}^{k-j}(z),\] where $Q_a(z) \coloneqq z^{-\phi_0(a)} \sum_{i=0}^{d_a} q_{i}^a z^i = \sum_{i=-\phi_0(a)}^{d_a-\phi_0(a)} q_{i+\phi_0(a)}^a z^i$, $q_{i}^a \in \mathbb{R}$ for all $i$.
    \item For all $a \in \mathbb{T}_k$, $d_a= (k-1)(p-1)$ so $Q_a(z)$ has degree $(k-1)(p-1)-\phi_0(a)$ and $q_{d_a}^a=(pa_p)^{k-1}$.
\end{enumerate}
\end{lemma}
\begin{proof}
 First, as $F_{\mu}$ satisfies \eqref{hp}, the $j^{th}$-derivative of $F^{j}_{\mu}$ can be written as $F^{j}_{\mu}(z) = P_j(z)F_{\mu}(z)$ with $P_j$ a polynomial function with degree $j(p-1)$. So, for all $a \in \mathbb{T}_k$, we have \[\prod_{u\in I(a)} F_{\mu}^{(c_u(a))}(z) = F_{\mu}^{k-\phi_0(a)}(z)\prod_{u\in I(a)} P_{c_u(a)}(z), \] \[\prod_{u\notin I(a)} \widetilde{F}_{\mu}(z) = z^{-\phi_0(a)}\bigl(F_{\mu}(z)-\mu(0) \bigr)^{\phi_0(a)} = z^{-\phi_0(a)}\sum_{j=0}^{\phi_0(a)}\binom{\phi_0(a)}{j}\mu(0)^jF^{\phi_0(a)-j}_{\mu}(z),\]
 thus by setting $Q_a(z) =  z^{-\phi_0(a)}\prod_{u\in I(a)} P_{c_u(a)}(z) \coloneqq z^{-\phi_0(a)} \sum_{i=0}^{d_a} q_{i}^a z^i $ we deduce that \[G^a_{\mu}(a) = Q_a(z)\sum_{j=0}^{\phi_0(a)}\binom{\phi_0(a)}{j}\mu(0)^jF^{\phi_0(a)-j}_{\mu}(z), \]
 so we get the first point. 
 
 Then, as $\deg(P_j)=j(p-1)$, we have that $d_a=\sum_{u \in I(a)} c_u(a)(p-1) = (k-1)(p-1)$ so $Q_a(z)$ has degree $(k-1)(p-1)-\phi_0(a)$. Moreover, the dominant term of $F^{j}_{\mu}$ is equal to the dominant term of $(P'(z))^j$ which is equal to $(pa_pz^{p-1})^j$, consequently $q_{d_a}^a=\prod_{u=1}^k(pa_p)^{c_u(a)} = (pa_p)^{k-1}.$
\end{proof}
For all $m\in \mathbb{Z}_{\geq0}$, we define $e^{m}_n\coloneqq [z^n]F^m_{\mu}(z)$. Applying the first part of Lemma \ref{lempolyetF} we deduce the following equality \begin{equation}\label{formule_coef_G_a}
    [z^{n-k-\phi_0(a)}]G^a_{\mu}(z)= \sum_{j=0}^{\phi_0(a)} \sum_{i=0}^{d_a-\phi_0(a)}\binom{\phi_0(a)}{j}\mu(0)^jq_{i+\phi_0(a)}^ae^{k-j}_{n-k-\phi_0(a)-i}. 
\end{equation}
The following lemma essentially tells that the coefficents in $F^k_{\mu}$ decay fairly quickly, which is typically observed in cases where the generating function has an infinite radius of convergence.

\begin{lemma}\label{cvg_e_n}
    Fix $m \in [\![1,k]\!]$. Then, \[\frac{e^{m}_{n+1}}{e^{m}_{n}} \underset{n \to \infty}{\sim} (mpa_p)^{1/p}n^{-1/p}.\]
\end{lemma}
\begin{proof}
    Fix $m\in [\![1,k]\!] $. By applying Corollary VIII.2 of \cite{F13} to $F_{\mu}^{m}(z) = c^{m} \exp(mP(z))$, we get that \begin{equation*}
        e^{m}_n \underset{n \to \infty}{\sim} \frac{c^{m} }{\sqrt{2\pi\lambda_n}}\frac{\exp{(mP(r_n))}}{r_n^n},
    \end{equation*}
    where $r_n$ and $\lambda_n$ satisfy $r_nP'(r_n)=n/m$ and $\lambda_n = m(n + r_n^2P^{(2)}(r_n)).$
    Thus, we have the following estimate 
    \begin{equation} \label{eqpois}
        \frac{ e^{m}_{n+1}}{e^{m}_n} \underset{n \to \infty}{\sim} \sqrt{\frac{\lambda_n}{\lambda_{n+1}}} \frac{r_n^n}{r_{n+1}^{n+1}}\exp{\bigl(m(P(r_{n+1})-P({r_n}))\bigr)}.
    \end{equation}
     In addition, we have that $r_nP'(r_n) \underset{n \to \infty}{\sim} r_npa_pr_n^{p-1}$ so we deduce that $r_n \underset{n \to \infty}{\sim} (n/mpa_p)^{1/p}$ and $\lambda_n \underset{n \to \infty}{\sim} mpn$. Therefore, \begin{equation} \label{behaviorlambdarn}
       \lambda_{n}\underset{n \to \infty}{\sim}\lambda_{n+1} \quad \text{and} \quad  r_n \underset{n \to \infty}{\sim} r_{n+1}.
    \end{equation} To study the limit behavior of $e^{m}_{n+1}/e^{m}_n$, Equation \eqref{behaviorlambdarn} provides sufficient information for $\lambda_{n}$, but a more refined asymptotic expansion is required for $r_n$. 
    By setting $b_i \coloneqq(i+1)a_{i+1}$, we have
    \[r_n = \frac{n}{mP'(r_n)}= \frac{n}{m} \biggl( \sum_{i=1}^{p-1} b_i r^i_n \biggr)^{-1}=\frac{n}{mb_{p-1}r^{p-1}_{n}}\biggl( 1+ \sum_{i=1}^{p-2} \frac{b_i}{b_{p-1}} r^{i-(p-1)}_n \biggr)^{-1},\]
    so \[r^p_n = \frac{n}{mb_{p-1}} \Bigl( 1+ \frac{b_{p-2}}{b_{p-1}} \frac{1}{r_n} + o\Bigl(\frac{1}{r_n} \Bigr)\Bigr)^{-1},\]
    and by doing an asymptotic expansion and reinjecting the equivalent $r_n \underset{n \to \infty}{\sim} (n/mpa_p)^{1/p}$, we get that \[ r_n = \frac{1}{(mpa_p)^{1/p}}n^{1/p} - \frac{m^{1/p}(p-1)a_{p-1}}{p^2a_p} +o( 1).\]
    In fact, one can show by iterating this bootstrapping argument, that  $r_n$ admits an asymptotic expansion of the following form 
    \begin{equation}\label{devasympr_n}
        r_n = c_{-1}n^{1/p} + c_0 + \frac{c_1}{n^{1/p}} + \cdots + \frac{c_{p-1}}{n^{(p-1)/p}} +o\Bigl( \frac{1}{n^{(p-1)/p}}\Bigr),
    \end{equation} 
where $c_{-1} = (mpa_p)^{-1/p}$, $c_0 =m^{1/p}(p-1)a_{p-1}/p^2a_p$ and $c_1,\ldots,c_{p-1}$ are real numbers. 
Consequently, using Equation \eqref{devasympr_n}, one can verify that $r_{n+1}^{n+1}/r_n^n \underset{n \to \infty}{\sim} (mpa_p)^{1/p}e^{-1/p} n^{-1/p}$ and $\exp{(m(P(r_{n+1})-P({r_n})))} \underset{n \to \infty}{\sim} e$. The lemma follows by injecting these estimates in Equation \eqref{eqpois} .
\end{proof}
We may now proceed to the proof of the Proposition \ref{proppois}.
\begin{proof}[\textit{Proof of Proposition }\ref{proppois}]
    Set  $a,b \in \mathbb{T}_k $. By Equation \eqref{formule_coef_G_a} and then by applying Lemma \ref{cvg_e_n}, we have \[[z^{n-k-\phi_0(a)}]G^a_{\mu}(z)= \sum_{j=0}^{\phi_0(a)}\sum_{i=0}^{d_a-\phi_0(a)}\binom{\phi_0(a)}{j}\mu(0)^jq_{i+\phi_0(a)}^ae^{k-j}_{n-k-\phi_0(a)-i}   \underset{n \to \infty}{\sim} \sum_{j=0}^{\phi_0(a)}\binom{\phi_0(a)}{j}\mu(0)^jq_{d_a}^ae^{k-j}_{n-k-d_a}. \]
Fix $m\in [\![1,k]\!] $. By \eqref{eqpois}, we have \[\frac{e^{m}_{n-k-d_a}}{e^{m+1}_{n-k-d_a}} = c \exp{\bigl(mP(r_{n-k-d_a})- (m+1)P(r_{n-k-d_a})\bigr)} = c \exp{\bigl(-P(r_{n-k-d_a})\bigr)} \xrightarrow[n \rightarrow \infty]{} 0. \]
Therefore, we deduce that
\[[z^{n-k-\phi_0(a)}]G^a_{\mu}(z)\underset{n \to \infty}{\sim} q_{d_a}^ae^{k}_{n-k-d_a}\]
As $ q_{d_a}^ae_{n-k-d_a} =  q_{d_b}^be_{n-k-d_b}$ by the second part of Lemma \ref{lempolyetF}, the result follows.
\end{proof}

We now proceed with the proof of the second statement in Theorem \ref{thPois-intro}.

\begin{proof}[Proof of Theorem \ref{thPois-intro}.2]
The $X^n_i$'s are i.i.d. and satisfies \eqref{hp} so the generating function of $X^n_2+ \cdots +X^n_k$ is equal to \[F_{X^n_2+ \cdots +X^n_k}(z)=F^{k-1}_{\mu}(z)=c^{k-1}\exp{\bigl((k-1)P(z)\bigr)}.\]
Let $\epsilon >0$ and define $n_x\coloneqq (n-1)(1/k+x)$ for $x \in \mathbb{R}$. Recalling that for all $m\in \mathbb{Z}_{\geq0}$, $e^{m}_n= [z^n]F^m_{\mu}(z)$, we have \[\mathbb{P}\Bigl( \Big|\frac{X^n_1 }{n-1}-\frac{1}{k} \Big| \geq \epsilon \,\Big| \,X^n_1+ \cdots +X^n_k=n-1\Bigr) = \frac{\sum_{|i-(n-1)/k|\geq \epsilon (n-1) }  e^{1}_i e^{k-1}_{n-1-i}}{\sum_{\substack{0\leq i,j \leq n-1 \\ i+j = n-1}} e^{1}_i e^{k-1}_j } = \frac{\sum_{|i-(n-1)/k|\geq \epsilon (n-1) } u^{i,k}_n}{\sum_{\substack{0\leq i,j \leq n-1 \\ i+j = n-1}} u^{i,k}_n },\] where we define $u^{i,k}_n \coloneqq e^{1}_i e^{k-1}_{n-1-i}$.
We also define $ A^{\epsilon,k,l}_n, A^{\epsilon,k,r}_n$ and $B^{\epsilon,k,l}_n,B^{\epsilon,k,r}_n$ by:
\[A^{\epsilon,k,l}_n = \sum_{i \leq n_{-\epsilon} } u^{i,k}_n, \quad  A^{\epsilon,k,r}_n  = \sum_{i \geq n_{\epsilon} } u^{i,k}_n, \quad B^{\epsilon,k,l}_n  = \sum_{n_{-\epsilon}<i <(n-1)/k } u^{i,k}_n \quad \text{and} \quad B^{\epsilon,k,r}_n= \sum_{(n-1)/k<i <n_{\epsilon}} u^{i,k}_n.\]
In particular, we have that 
\begin{equation}\label{ineqAB}
    \mathbb{P}\Bigl( \Big|\frac{X^n_1 }{n-1}-\frac{1}{k} \Big| \geq \epsilon \,\Big| \,X^n_1+ \cdots +X^n_k=n-1\Bigr) \leq \frac{A^{\epsilon,k,l}_n}{B^{\epsilon,k,l}_n} + \frac{A^{\epsilon,k,r}_n}{B^{\epsilon,k,r}_n},
\end{equation}
so we now aim to prove that $A^{\epsilon,k,l}_n/B^{\epsilon,k,l}_n$ and $A^{\epsilon,k,r}_n/B^{\epsilon,k,r}_n$ tends to zero when $n$ tends to infinity to conclude the proof. Fix $\eta >0.$
By Lemma \ref{cvg_e_n}, we have for all $m \in \mathbb{Z}_{>0}$:
\begin{equation}\label{eq_enm}
   \frac{e^{m}_{n+1}}{e^{m}_{n}} =(mpa_p)^{1/p}n^{-1/p}\bigl(1+\delta(m,n) \bigr), 
\end{equation}
where $\delta(m,n)/n  \underset{n \to \infty}{\rightarrow} 0 $. Therefore, for each $m$, we choose  $N_m$ such that for all $n\geq N_m$, $|\delta(m,n)| \leq \eta $. Additionally, let $N'_m$ be an integer such that $(e^m_n)_{n\geq N'_m}$ is a nonincreasing sequence. We first show that $A^{\epsilon,k,r}_n/B^{\epsilon,k,r}_n  \underset{n \to \infty}{\rightarrow}0.$ Let $n-1\geq \max{(kN_1,N'_1+N_{k-1})}$. We define $n_\epsilon\coloneqq (n-1)(1/k+\epsilon)$. Then, for every $i \in [\![n_{\epsilon/2},n-1-N_{k-1}]\!]$, we have by Equation \eqref{eq_enm}
\[\frac{u^{i+1,k}_n}{u^{i,k}_n}= \frac{e^{1}_{i+1}e^{k-1}_{n-2-i}}{e^{1}_{i}e^{k-1}_{n-1-i}} \leq \biggl(\frac{n-2-i}{(k-1)i}\biggr)^{1/p} \frac{1+\eta}{1-\eta}.\]
As $x\mapsto (n-2-x/(k-1)x)^{1/p}$ is a decreasing function, we deduce that
\[\frac{u^{i+1,k}_n}{u^{i,k}_n}\leq \biggl(\frac{n-2-n_{\epsilon/2}}{(k-1)n_{\epsilon/2}}\biggr)^{1/p} \frac{1+\eta}{1-\eta},\]
where $0<(n-2-n_{\epsilon/2}/(k-1)n_{\epsilon/2})^{1/p}<1$ so that for sufficiently small $\eta$, there exists $\alpha_\epsilon <1$ such that 
\begin{equation}\label{ineg_u_n}
    \frac{u^{i+1,k}_n}{u^{i,k}_n}\leq \alpha_\epsilon.
\end{equation}
Now, for the remaining $N_{k-1}$ values of $i \in [\![n-N_{k-1},n-1]\!]$, i.e.  $N_{k-1}$ values of $j\coloneqq n-1-i \in [\![0,N_{k-1}-1]\!]$,  we notice that
\[\frac{u_n^{i,k}}{u_n^{n-1-N_{k-1},k}}= \frac{e^{1}_{i}e^{k-1}_{n-1-i}}{e^{1}_{n-1-N_{k-1}}e^{k-1}_{N_{k-1}}}= \frac{e^{1}_{n-1-j}e^{k-1}_{j}}{e^{1}_{n-1-N_{k-1}}e^{k-1}_{N_{k-1}}}\]
where $e^{k-1}_{j}/e^{k-1}_{N_{k-1}} \leq \max{\{e^{k-1}_{j}/e^{k-1}_{N_{k-1}} : 0 \leq j \leq N_{k-1}-1 \} } \coloneqq c_k$ and $e^{1}_{n-1-j}/e^{1}_{n-1-N_{k-1}} \leq 1$. Therefore, for every $i \in [\![n-N_{k-1},n-1]\!]$, $u_n^{i,k} \leq c_ku_n^{n-1-N_{k-1},k}\leq c_k \alpha_\epsilon^{n-1-N_{k-1}-n_\epsilon} u_n^{n_\epsilon,k} \leq c_k u_n^{n_\epsilon,k}$, where the second inequality comes from Equation \eqref{ineg_u_n}. 
Hence, applying \eqref{ineg_u_n} for $i \in [\![n_{\epsilon},n-1-N_{k-1}]\!]$, we get that \[A^{\epsilon,k,r}_n = \sum_{n_\epsilon\leq i\leq n-1 } u^{i,k}_n \leq  \sum_{n_\epsilon \leq i\leq n-1-N_{k-1} } u^{n_\epsilon,k}_n \alpha_\epsilon^{i-n_\epsilon} +  \sum_{n-N_{k-1}\leq i\leq n-1}u^{i,k}_n \leq \frac{u^{n_\epsilon,k}_n}{1-\alpha_\epsilon} + c_kN_{k-1}u_n^{n_\epsilon,k}\]
and
\[B^{\epsilon,k,r}_n = \sum_{(n-1)/k<i <n_\epsilon} u^{i,k}_n \geq \sum_{n_{\epsilon/2}<i <n_\epsilon} u^{i,k}_n \geq \frac{\epsilon}{2}(n-1)u^{n_\epsilon,k}_n. \]
The last inequality follows from the fact that, for $i\geq n_{\epsilon/2},$ $u^{i,k}_n$ is nonincreasing in $i$ as stated in Equation \eqref{ineg_u_n}. 
Consequently, \[\frac{A^{\epsilon,k,r}_n}{B^{\epsilon,k,r}_n} \leq \frac{2}{(1-\alpha_\epsilon)\epsilon (n-1) } + \frac{2c_kN_{k-1}}{\epsilon (n-1)} \xrightarrow[n \rightarrow \infty]{} 0.\]
Applying the same method, we obtain that 
$A^{\epsilon,k,l}_n/B^{\epsilon,k,l}_n  \underset{n \to \infty}{\rightarrow}0$, and thus the result follows by inequality \eqref{ineqAB}.
\end{proof}
\bibliographystyle{alpha}
\bibliography{sample}

\end{document}